\documentclass[12pt,oneside]{amsart}
\usepackage{amscd,amssymb,multirow,amsmath}
\usepackage{hyperref}

\usepackage{typearea}
\usepackage{tikz,longtable}

\newcounter{lemmacounter}
\newcounter{thmcounter}

\newtheorem{lemma}[lemmacounter]{Lemma}
\newtheorem{theorem}[thmcounter]{Theorem}

\newcommand{\IC}{{\bf C}}

\newcommand{\IQbar}{\overline{\bf Q}}
\newcommand{\IZ}{{\bf Z}}
\newcommand{\IN}{{\bf N}}

\newcommand{\IF}{{\bf F}}

\newcommand{\IQ}{\mathbf{Q}}

\newcommand{\ssm}{\smallsetminus}

\newcommand{\legendre}{t}
\newcommand{\cN}{\mathcal N}

\newcommand{\cP}{\mathcal P}

\newcommand{\Fh}[1]{h_F{({#1})}}
\newcommand{\height}[1]{h{({#1})}}
\newcommand{\imag}[1]{{\rm Im}{({#1})}}

\newcommand{\ex}[1]{{ex}({#1})}

\makeatletter
\def\imod#1{\allowbreak\mkern10mu({\operator@font mod}\,\,#1)}
\makeatother

\newcommand {\diff}[2] {\frac{\partial #1}{\partial #2}}

\begin{document}
  \title{Six unlikely intersection problems in search of effectivity}
\author[P. Habegger]{P. Habegger}
\address[P. Habegger]{University of Basel, Spiegelgasse 1, 4051 Basel,
Switzerland}
\email{philipp.habegger\makeatletter @\makeatother unibas.ch}

\author[G. Jones]{G. Jones}
\address[G. Jones]{School of Mathematics, University of Manchester, Oxford Road,
Manchester, M13 9PL, UK}
\email{gareth.jones-3\makeatletter @\makeatother manchester.ac.uk}

\author[D. Masser]{D. Masser}
\address[D. Masser]{University of Basel, Spiegelgasse 1, 4051 Basel,
Switzerland} 
\email{david.masser\makeatletter @\makeatother unibas.ch}

\maketitle

\begin{abstract}
  We investigate four properties related to an elliptic curve $E_t$ in
  Legendre form with parameter $t$: the curve $E_{t}$ has  complex
  multiplication, $E_{-t}$ has complex multiplication, a point on
  $E_t$ with abscissa $2$ is of finite order, and $t$ is a root of
  unity. Combining all pairs of properties leads to six problems on unlikely
  intersections. We solve these problems effectively and in certain cases also
  explicitly. 
\end{abstract}
\section{Introduction}


For any $\legendre\in \IQbar\ssm\{0,1\}$ we consider the elliptic curve $E_\legendre$
\begin{equation}
\label{eq:legendre}
  y^2 = x(x-1)(x-\legendre)
\end{equation}
whose $j$-invariant equals
\begin{equation}
\label{eq:jlambda}
  j = 2^8 \frac{(\legendre^2-\legendre+1)^3}{\legendre^2 (1-\legendre)^2}.
\end{equation}





We will investigate four properties of $t$.

The first is that $E_t$ should have complex multiplication. The second, not a lot different, is that $E_{-t}$ should have complex multiplication. The third is that the point $P_t=(2, \sqrt{4-2t})$ should be torsion on $E_t$. And the fourth is the simpler exponential analogue: $t$ should be a root of unity.

It is well known that each property holds for infinitely many values of $t$. This is classical for the first two properties, was proved by Masser and Zannier in \cite{MZ:torsionanomalousAJ} for the third property, and is trivial for the fourth property.

We obtain six problems by looking at all pairs of properties.  
The reader is invited to imagine a regular tetrahedron with horizontal base whose vertices are the first, second, third properties, and whose apex is the fourth property. Thus the problems correspond to the six edges.

It is known that for any pair there are at most finitely many $t$ having both properties. For the pair $\{ 1,2\}$ of the first and second property, this reduces to a special case of the Andr\'e-Oort conjecture for a certain plane curve (of degree $4$), and this conjecture was proved by Andr\'e \cite{Andre98} for all plane curves

For the pair $\{ 1,3\}$ this is a special case of a result of Andr\'e \cite{Andre01} when the complex multiplication order is maximal and in general by Pila \cite{Pila:AOMM09}.

For $\{ 1,4\}$ the result is a very special case of Pila's result \cite{Pila:AO} (as indeed the results $\{ 1,2\}, \{2,4\}$ are as well).

The finiteness in $\{ 2,3\}$ reduces to the point $(2, \sqrt{4+2t})$ on $E_t$ and so the finiteness follows as in $\{ 1,3\}$.

The finiteness in $\{ 2,4\}$ is equivalent to that in $\{1,4\}$ because $-t$ is a root of unity if and only if $t$ is.

Finally, the finiteness in $\{ 3,4\}$ can be deduced from the result of Masser and Zannier \cite{MasserZannierMathAnn} by taking complex conjugates to deduce that $\overline{P_t}=P_{\overline t}=P_{1/t} = (2, \sqrt{ 4- 2/t})$ is torsion on $E_{1/t}$, and noting that $E_t,E_{1/t}$ are isomorphic to get another point $Q_t=(2t,t\sqrt{4t-2})$ on $E_t$. One must check that $P_{t},Q_t$ are generically linearly independent; but this follows easily from the fact that $P_t$ is ramified only at $t=2$ and $Q_t$ only at $t=1/2$.

All these proofs have various elements of possible ineffectivity which
make it difficult or impossible actually to find the finite set in
question. Thus in \cite{Andre98}, \cite{Pila:AOMM09}, and \cite{Pila:AO} there is an appeal
to Siegel-Brauer on class numbers, ineffective to this day. And in
\cite{Andre01} results of Silverman on bounded height and Colmez on unbounded
height are used, which involve unspecified constants. Finally in \cite{MasserZannierMathAnn} (which also uses Silverman's result) a key role is played by Pila's estimates for rational points on analytic surfaces which involve sophisticated compactness results of Gabrielov.

Since these proofs, there has been some progress towards effectivity. Regarding our problem pairs, we now know the following.

For $\{ 1,2\}$ effective versions of Andr\'e's Theorem \cite{Andre98} were proved independently by K\"uhne \cite{K1} and Bilu, Masser and Zannier \cite{BiMaZa}. The effectivity was illustrated by the line $x+y=1$ and the hyperbola $xy=1$; however our curve is much more complicated. We succeed here to show in Theorem \ref{thm:ag} that there are no $t$ with a bit of what looks like luck.

For $\{ 1,3\}$ we use a method involving supersingular primes and the
Chebotarev Density Theorem. 
We show in Theorem \ref{thm:abgeneral} that the finite set consists exactly of $t=2$ and the two roots of $t^2-16t+16 =0$.

For $\{1,4\}$ the effectivity follows from work of
 Paulin \cite{Paulin:effAO,Paulin:effAOlinforms}, whose result even
 implies that the order of $t$ is at most $2346$. 
 Here we show in Theorem \ref{thm:ac0} that $t=1 $ or $e^{\pm  \pi i /3}$.

For $\{ 2,3\}$ we could hope to proceed as in $\{ 1,3\}$, but it is not clear if this works.
  Instead we return to the method of height comparison; we hope that
 our methods of calculating the constants may have other
 applications. Here we prove in Theorem \ref{thm:bg} 
 for example that ${2^8(t^2+t+1)^3\over t^2(t+1)^2}=j(\tau)$ for the modular function, with $a\tau^2+b\tau+c=0$ for integers $a,b,c$ satisfying
$$0\leq b \leq a \leq \sqrt{-D/3},~~~b^2-4ac=-D$$
and $0<-D<2\cdot 10^{32}$.

As mentioned above we can forget about $\{2,4\}$.

Finally for $\{ 3,4\}$ we use an idea of Boxall and
 Jones \cite{BoxallJones} involving transcendence measures to obtain a
 zero estimate which can then be fed into the Bombieri-Pila machinery
 on analytic curves (also not entirely luck-free). Here we show in
 Theorem  \ref{thm:bc0} 
 only that the order of $t$ can be bounded in an effective way.

In connexion with the third property one must mention the surprising result of Stoll \cite{Stoll} that there are no complex numbers $t \neq 0,1$ such that $(2, \sqrt{4-2t})$ and $(3, \sqrt{18-6t})$ are both torsion (the original unlikely intersection proposed in \cite{MZ:torsionanomalousAJ}).

One can also consult W\"ustholz \cite{Wustholz} for interesting discussions about the general topic of effectivity and in particular the problems of inverting the modular function.


\section{$E_t$ and $E_{-t}$ have CM $\{1,2\}$}
\label{sec:ag}
We consider the pair $\{ 1,2\}$. We prove the following.
\begin{theorem}
\label{thm:ag}
 There are no complex numbers $t \neq 0,1,-1$ such that the
 elliptic curves $E_t$ and $E_{-t}$
both have complex multiplication.
\end{theorem}

We note for comparison that $E_t,E_{1-t}$ are isomorphic, so there are infinitely many $t$ such that they both have complex multiplication.

Because the $j$-invariants are respectively
\begin{equation}
\label{eq:old1}
x=256{(t^2-t+1)^3 \over t^2(1-t)^2},\quad y=256{(t^2+t+1)^3 \over t^2(1+t)^2}
\end{equation}
which are related by
\begin{alignat}1
\label{eq:old2}
 &x^3y-2x^2y^2+xy^3-1728x^3+1216x^2y+1216xy^2-1728y^3+3538944x^2 \\
\nonumber
 &-2752512xy+3538944y^2-2415919104x-2415919104y+549755813888=0
\end{alignat}
this is implied by the fact that (\ref{eq:old2}) has no solutions in singular moduli $x,y$; that is, values of the $j$-function $j(\tau)$ at quadratic numbers $\tau$ in the upper half plane.

Actually the two facts are equivalent, because (\ref{eq:old1}) for $t\neq 0,1,-1$ parametrizes the entire affine curve (\ref{eq:old2}); i.e. there are no missing points. To see this we use the identity
\begin{equation}
\label{eq:old3}
F(R(t),y)=S(t)^2(y-R(-t))(y-R(-t'))(y-R(-t''))
\end{equation}
where $F(x,y)$ is the left-hand side of (\ref{eq:old2}),
$$R(t)=256{(t^2-t+1)^3 \over t^2(1-t)^2}$$
as in (\ref{eq:old1}),
$$S(t)=8{(2t-1)(t-2)(t+1) \over t(1-t)},$$
and $t'=1-t,~t''=(t-1)/t$. For any point $(x,y)$ on (\ref{eq:old2}) we can write $x=R(t)$ for $t\neq 0,1$. If $x=1728=R(1/2)=R(2)=R(-1)$ then $y=21952/9=R(-1/2)$, so we can assume $t \neq 1/2,2,-1$. It now follows from (\ref{eq:old3}) that $y$ is one of $R(-t),R(-t'),R(-t'')$. But at the same time $R(t)=R(t')=R(t'')$.

Andr\'e \cite{Andre98} determined when when there are only finitely many singular moduli $x,y$ satisfying a given polynomial equation $\mathcal{F}(x,y)=0$ and when not. 
His proof did not allow an effective determination of all solutions due to the use of the Siegel-Brauer Theorem. This ingredient was eliminated by K\"uhne \cite{K1} in 2012 and independently by Bilu, Masser and Zannier \cite{BiMaZa} in 2013 thus yielding effectivity. This was illustrated by a second paper \cite{K2} of K\"uhne showing that there are no solutions on the line $x+y=1$. And in \cite{BiMaZa} it was also shown, with a bit more trouble, that there are no solutions on the hyperbola $xy=1$.

At first sight it looks like (\ref{eq:old2}), now our $F(x,y)=0$, should be a lot more trouble. But happily it turns out that certain simplifying features for the line and the hyperbola persist for (\ref{eq:old2}). In general the difficulty depends crucially on the points at infinity on the curve ${\mathcal F}(x,y)=0$. For $x+y=1$ there is only one, corresponding to the asymptote $y=-x$. A general $y=\alpha x$ would require the use of linear forms in two logarithms; but if $\alpha$ is a root of unity then we can get by with a single logarithm, which comes down to a simple Liouville-type estimate for an algebraic number. For $xy=1$ there is an asymptote $y=0$. A general $y=\alpha$ would require the use of linear forms in two elliptic logarithms, with some resulting computational troubles; but if $\alpha$ happens itself to be a singular modulus then again we can get by with Liouville. As $0=j((1+\sqrt{-3})/2)$ this is the case here, and there remain only minor ramification troubles in inverting $j(\tau)$ near $\tau=(1+\sqrt{-3})/2$ (which has a zero of order 3 there).

For (\ref{eq:old2}) it is not absolutely obvious at first sight what happens near infinity. By symmetry we can assume $x$ large. By (\ref{eq:old1}) this implies $t$ is near $0,1,\infty$. Then $y$ is near $\infty,1728,\infty$ respectively. For the first of these $x,y$ are both about $256/t^2$, so the asymptote is $y=x$. For the third $x,y$ are both about $256t^2$, so again $y=x$. So no logarithms. As for the second, we see that $x$ is about $256/(1-t)^2$ and $y$ is about $1728$, so the asymptote is $y=1728$; and lo and behold $1728=j(\sqrt{-1})$. So no elliptic logarithms either. But the ramification persists, now of order 2 rather than 3 above.

We start with a study of the rational function $R(t)$. It is well-known to be invariant under the group generated by the transformations taking $t$ to $1/t$ and $1-t$.

We can check that for real $t$ it is monotonic on the six open
intervals
\begin{equation}
\label{eq:old4}
(-\infty,-1),~(-1,0),~(0,1/2),~(1/2,1),~(1,2),~(2,\infty).
\end{equation}
(alternately decreasing and increasing) with local minima at $t_0=-1,1/2,2$ when $R(t_0)=1728$. Thus if $x>1728$ then there are exactly six different real $t$ with $R(t)=x$, one in each of the intervals (\ref{eq:old4}); and there are no non-real $t$ with $R(t)=x$. The next result specifies $y=R(-t)$ more precisely in each of the intervals when $x \geq 1000000$ (among other things).
\bigskip
\noindent

\begin{lemma}
\label{lem:david1}
 Suppose $R(t) \geq 1000000$.
 \begin{enumerate}
 \item [(i)] 
If $t < -1$ then $R(-t) \geq R(t)-32\sqrt{R(t)}$.
\item[(ii)] 
If $-1 <  t < 0$ then $R(-t) \geq R(t)-32\sqrt{R(t)}$.
\item[(iii)] 
If $0 < t  < 1/2$ then $R(-t) \geq R(t)$.
\item[(iv)]  
 If $1/2 < t < 1$ then $R(t) \leq 265(t-1)^{-2}$ and $0<R(-t)-1728 \leq 1363(t-1)^{2}$.
\item[(v)] 
If $1<t<2$ then $R(t) \leq 270(t-1)^{-2}$ and $0<R(-t)-1728 \leq 1340(t-1)^{2}$.
\item[(vi)] 
 If $t>2$ then $R(-t) \geq R(t)$.
 \end{enumerate}
 \end{lemma}
 \begin{proof}
  In (i) we must have $t \leq -60$, else $R(t)<R(-60)<1000000$. We calculate
$$1024R(t)-(R(-t)-R(t))^2=-262144{P(t) \over t(t-1)^4(t+1)^4}$$
for a polynomial $P(t)=t^{10}+\cdots$ with no real zero $t \leq -60$. Thus $P(t)>0$ for such $t$, and the result follows.

In (ii) we get the result immediately from (i)  by replacing $t$ by $1/t$.

In (iii) we have $t \leq 1/60$, else $R(t) < R(1/60)=R(60)<1000000$. Now
$$R(-t)-R(t)=512{P(t) \over t(t-1)^2(t+1)^2}$$
for a polynomial $P(t)$ with $P(0)>0$ and no positive real zero $t \leq 1/60$. Thus $P(t)>0$ for such $t$, and the result follows.

In (iv) we have $t \geq 1-1/60$, else $R(t) < R(1-1/60)=R(60)$. Now
$$R(t)=R(1-t)=256{(1-(1-t)+(1-t)^2)^3 \over t^2(1-t)^2}\leq 256{1 \over (1-1/60)^2(1-t)^2}<{265 \over (1-t)^2}$$
and
$$R(-t)-1728=64{(t+2)^2(2t+1)^2(t-1)^2 \over t^2(t+1)^2}<5184{(t-1)^2 \over (1-1/60)^2(2-1/60)^2}<1363(t-1)^2.$$

In (v) we have $t \leq 1+1/60$, else $R(t) < R(1+1/60)<1000000$. Now
$$R(t)=256{(t^2-t+1)^3 \over t^2(1-t)^2}\leq 256{((1+1/60)^2-(1+1/60)+1)^3 \over (1-t)^2}<{270\over (1-t)^2}$$
and
$$R(-t)-1728=64{(t+2)^2(2t+1)^2(t-1)^2 \over t^2(t+1)^2}<16(3+1/60)^2(3+1/30)^2(t-1)^2<1340(t-1)^2.$$

In (vi) we get the result immediately from (iii) using by replacing $t$ by $1/t$.

This completes the proof.
\end{proof}

We will also have to solve $R(t)=x$ for $x<1728$. Here there are no
real solutions. If $x<0$ there are still six different complex
solutions, and they have real part $1/2$ or lie on the circles $|t|=1$ and $|t-1|=1$. This can be seen from
$$R(1/2+iv)=-64{(4v^2-3)^3 \over (4v^2+1)^2},$$
which is monotonic increasing on $(-\infty,0)$, monotonic decreasing on $(0,\infty)$ and zero at $v=\pm\sqrt{3}/2$ (with graph looking a bit like a buddhist temple); thus there are solutions with $v>\sqrt{3}/2$ and $v<-\sqrt{3}/2$. And from
$$R(u+i\sqrt{1-u^2})=128{(2u-1)^3 \over u-1}$$
which is monotonic decreasing on $(-1,1)$ (with each choice of sign for the square root) and zero at $u=1/2$; so there are two solutions with $1/2<u<1$. And
$$R(1+u+i\sqrt{1-u^2})=128{(2u+1)^3 \over u+1}$$
which is similarly monotonic increasing on $(-1,1)$ and zero at $u=-1/2$; so there are two solutions with $-1<u<-1/2$.
\begin{lemma}
\label{lem:david2}
 Suppose $R(t) \leq -190000$.
 \begin{enumerate}
 \item [(i)] 
 If $t=1/2+iv$ with $v>\sqrt{3}/2$ then $|R(-t)| \geq -R(t)$.

 \item [(ii)] 
If $t=1/2+iv$ with $v<-\sqrt{3}/2$ then $|R(-t)| \geq -R(t)$.

 \item [(iii)] 
If $t=u+i\sqrt{1-u^2}$ with $u>1/2$ and $\sqrt{1-u^2}>0$ then $R(t) \geq -128(1-u)^{-1}$ and $-2593(1-u) \leq R(-t)-1728<0$.

 \item [(iv)] 
If $t=u+i\sqrt{1-u^2}$ with $u>1/2$ and $\sqrt{1-u^2}<0$ then $R(t) \geq -128(1-u)^{-1}$ and $-2593(1-u) \leq R(-t)-1728<0$.

 \item [(v)] 
If $t=1+u+i\sqrt{1-u^2}$ with $u<-1/2$ and $\sqrt{1-u^2}>0$ then $|R(-t)| \geq -R(t)$.

 \item [(vi)] 
If $t=1+u+i\sqrt{1-u^2}$ with $u<-1/2$ and
 $\sqrt{1-u^2}<0$ then $|R(-t)| \geq -R(t)$.
 \end{enumerate}
 \end{lemma}
 \begin{proof}
  In (i), (ii) we find
$$|R(-t)|^2-R(t)^2=32768{P(v) \over (4v^2+1)^4(4v^2+9)^2}$$
for a polynomial $P(v)$ with non-negative coefficients, so the results are clear.

In (iii) we have $u \geq u_0=1-1/1450$, else $R(u+i\sqrt{1-u^2})>R(u_0+i\sqrt{1-u_0^2})>-190000.$ Now
$$-R(t)=128{(2u-1)^3 \over 1-u}\leq128(1-u)^{-1}$$
and
$$1728-R(-t)=64{(1-u)(5+4u)^2 \over u+1}\leq 5184{1-u \over u_0+1} \leq 2593(1-u).$$

Part (iv) follows by complex conjugation.

In (v), (vi) we have
$$|R(-t)|^2-R(t)^2=-65536{P(u) \over (u+1)(4u+5)^2}$$
for a polynomial $P(u)$ with no real zeroes between $-1$ and $-1/2$
and $P(-3/4)<0$, so the results follow.

\end{proof}

Next comes the study of $j(\tau)$ itself.

\begin{lemma}
\label{lem:david3}
If $\tau$ is in the standard fundamental domain with imaginary part
$y$ then $||j(\tau)|-e^{2\pi y}| \leq 2079$.
\end{lemma}
\begin{proof}
This is Lemma 1 of \cite{BiMaZa}. It is essentially inversion of $j(\tau)$ near $\tau=\infty$.
\end{proof}

Now we have the crucial inversion near $\tau=i$, first only on the
imaginary axis.

\begin{lemma}
\label{lem:david4}
 If $\tau=iy$ for $y \geq 1$ and $0 < j(\tau)-1728 \leq \delta \leq 2$ then $0 <y-1 \leq \sqrt{\delta}/150$.
 \end{lemma}
 \begin{proof}
  We note that $y \leq y_0=101/100$ else $j(\tau) > j(iy_0) > 1730$. Now for $f(y)=j(iy)=e^{2\pi y}+\cdots$ we have
$$f''(y)=(2\pi)^2e^{2\pi y}+\sum_{n=1}^\infty c_n(2\pi n)^2e^{-2\pi ny}$$
for $c_n \geq 0$. This is at least
$$(2\pi)^2e^{2\pi}+\sum_{n=1}^\infty c_n(2\pi n)^2e^{-2\pi ny_0}=(2\pi)^2e^{2\pi}+f''(y_0)-(2\pi)^2e^{2\pi y_0}.$$
A calculation shows that $f''(y_0)> 48364$ - it is most easily done using
$${d^2j\over d\tau^2}=-1152\pi^2{E_4(4E_6^2+3E_4^3-E_2E_4E_6) \over E_4^3-E_6^2}$$
for the standard modular forms $E_4,E_6$ and the related $E_2$
(definitely not Legendre curves) 
Thus $f''(y) \geq 46993$. And now
$$\delta \geq j(\tau)-1728 = f(y)-f(1) ={1 \over 2}f''(\eta)(y-1)^2$$
for $1 < \eta < y$, and the result follows.
\end{proof}

And finally the inversion near $\tau=i$ on the unit circle.

\begin{lemma}
\label{lem:david5}
If $\tau=e^{i\theta}$ for $\pi/3 \leq \theta \leq 2\pi/3$ and $-2 \leq
-\delta \leq j(\tau)-1728\leq 0$ then
$|\theta-\pi/2| \leq \sqrt{\delta}/100$.
\end{lemma}
\begin{proof}
The analogous proof for $|j(\tau)| \leq \delta$ in \cite{BiMaZa} was
heavily geometrical, especially in Lemma \ref{lem:david2}, %
 because we were working in the complex plane. Here we restrict to real analysis. After all, $f(\theta)=j(e^{i\theta})$ is real if $\theta$ is, thanks to
$$\overline{f(\theta)}=\overline{j(e^{i\theta})}=j(-e^{-i\theta})=j(e^{i\theta})=f(\theta).$$
As $j'(i)=0$ the leading term in the power series of $j(e^{i\theta})-1728$ about $\theta=\pi/2$ involves $f''$. To evaluate this derivative we use again $E_4,E_6$ and $E_2$, for which the analogous quantities
$$e_4(\theta)=e^{2i\theta}E_4(e^{i\theta}),\quad e_6(\theta)=e^{3i\theta}E_6(e^{i\theta})$$
are also real. And adjusting $E_2(\tau)$ by $3/(\pi\imag{\tau})$ in the usual way we see too that
$$e_2(\theta)=e^{i\theta}\left(E_2(e^{i\theta})-{3 \over \pi\sin \theta}\right)$$
is real.

We now start estimating (as if we had been doing something entirely different for the last five pages).

On the fundamental domain we have
$$|E_4(\tau)| \leq 1+240\sum_{n=1}^\infty\sigma_3(n)e^{-n\pi\sqrt{3}}=C_4=E_4\left({i\sqrt{3}\over 2}\right)$$
and rather similarly
$$|E_6(\tau)| \leq 1+504\sum_{n=1}^\infty\sigma_5(n)e^{-n\pi\sqrt{3}}=C_6=2-E_6\left({i\sqrt{3} \over 2}\right),$$
$$|E_2(\tau)| \leq C_2=2-E_2\left({i\sqrt{3}\over 2}\right).$$
It follows that for $E_4'=2\pi i(E_2E_4-E_6)/3$ we have
$$|E_4'(\tau)| \leq C_4'=2\pi (C_2C_4+C_6)/3,$$
for $E_6'=\pi i(E_2E_6-E_4^2)$ that
$$|E_6'(\tau)| \leq C_6'=\pi (C_2C_6+C_4^2),$$
and for $E_2'=\pi i(E_2^2-E_4)/6$ that
$$|E_2'(\tau)| \leq C_2'=\pi (C_2^2+C_4)/6.$$
So for real $\theta$ with $\pi/3 \leq \theta \leq 2\pi/3$ we deduce
$$|e_4'(\theta)| =|2ie^{2i\theta}E_4(e^{i\theta})+ie^{3i\theta}E_4'(e^{i\theta})|\leq c_4'=2C_4+C_4',$$
$$|e_6'(\theta)| \leq c_6'=3C_6+C_6',$$
and for
$$e_2'(\theta)=ie^{i\theta}E_2(e^{i\theta})+ie^{2i\theta}E_2'(e^{i\theta})+{3 \over \pi\sin^2\theta}$$
we get
$$|e_2'(\theta)| \leq c_2'=C_2+C_2'+{4 \over \pi}.$$

A calculation yields
\begin{equation}
\label{eq:old5}
f''=3456\pi^2e_4-{1152\pi e_4e_6(3ce_4+\pi se_2e_4-7\pi
se_6) \over s(e_4^3-e_6^2)}
\end{equation}
with $c=\cos\theta,~s=\sin\theta$.

At $\theta=\pi/2$ we have the well-known values
\begin{equation}
\label{eq:old6}
e_2=0,\quad e_4=-E_4(i)=-{3\Gamma(1/4)^8 \over (2\pi)^6},\quad e_6=0.
\end{equation}
Thus
\begin{equation}
\label{eq:old7}
f''\left({\pi \over 2}\right)=3456\pi^2e_4<-49655
\end{equation}
coming from the first term in (\ref{eq:old5}). We show that the other terms in (\ref{eq:old5}) are relatively small.

Now if $|f(\theta)-1728| \leq 2$ as in the present lemma, then by
$f(\pi/2+1/100)<1726$ and monotonicity
(because $f’ \neq 0$ or equivalently $j’ \neq 0$)
 we deduce
$|\theta-\pi/2| \leq 1/100$. For such $\theta$ we have
\begin{equation}
\label{eq:old8}
\left|e_4(\theta)-e_4\left({\pi \over 2}\right)\right| \leq {c_4'\over100},\quad |e_6(\theta)| \leq
 {c_6'\over100},\quad |e_2(\theta)| \leq {c_2'\over 100}
\end{equation}
and so from (\ref{eq:old7}) we lose at most
\begin{equation}
\label{eq:old9}
3456\pi^2{c_4'\over100} < 5569.
\end{equation}
For the other three terms in (\ref{eq:old5}) we note first from (\ref{eq:old8}) (see (\ref{eq:old6}) also) that
$$|e_4(\theta)^3-e_6(\theta)^2|\geq \left|\left|e_4\left({\pi\over2}\right)\right|-{c_4'\over100}\right|^3-\left({c_6'\over100}\right)^2= \left(-e_4\left({\pi\over2}\right)-{c_4'\over100}\right)^3-\left({c_6'\over100}\right)^2\geq 2.$$

The first of these terms, which is small because of $e_6$, can now be
estimated by
\begin{equation}
\label{eq:old10}
\left. {1 \over 2}1152\pi C_4\left(3C_4\left({c_6'\over100}\right)\right)
 \middle/  \sin \left({\pi\over2}+{1\over100}\right)<8537 \right.
\end{equation}
The second, which is small due to both $e_2$ and $e_6$, by
\begin{equation}
\label{eq:old11}
{1 \over 2}1152\pi C_4\left(\pi\left({c_2'\over100}\right)\right)C_4\left({c_6'\over100}\right)<368,
\end{equation}
and the third, also doubly small, by
\begin{equation}
\label{eq:old12}
{1 \over 2}1152\pi C_4\left(7\pi\left({c_6'\over100}\right)^2\right)<10915,
\end{equation}
a bit of a shock but fine.

We deduce from (\ref{eq:old7}), (\ref{eq:old9}), (\ref{eq:old10}), (\ref{eq:old11}), (\ref{eq:old12}) that for $|\theta-\pi/2| \leq 1/100$
$$|f''(\theta)| \geq 49655-5569-8537-368-10915=24266.$$
And finally
$$\delta \geq \left|f(\theta)-f\left({\pi\over2}\right)\right| = {1 \over 2}|f''(\phi)|\left|\theta-{\pi\over2}\right|^2 \geq {1 \over 2}24266\left|\theta-{\pi\over2}\right|^2$$
leading to the result.
\end{proof}

We can now tackle (\ref{eq:old2}).

We write
$$x=j(\tau)=R(t),\quad y=j(\sigma)=R(-t)$$
with $D,E$ as the respective discriminants of $\tau,\sigma$. We can assume
$$|D| \geq |E|$$
and that $\tau,\sigma$ are in the fundamental domain, with
$$\sigma={b+\sqrt{E} \over 2a}$$
so
$$|b| \leq a \leq \sqrt{{|E|\over 3}}.$$
There are two cases according to the parity of $D$.

First suppose $D$ is even, $|D| \geq 20$. Then by conjugating we can assume $\tau=\sqrt{D}/2$. By Lemma \ref{lem:david3}
$$R(t)=j(\tau)=|j(\tau)|\geq e^{\pi\sqrt{|D|}}-2079\geq e^{\pi\sqrt{20}}-2079>1000000.$$
Now $t$ is real and Lemma \ref{lem:david1} gives six subcases.

If $t<-1$ we get
\begin{equation}
\label{eq:old13}
y \geq e^{\pi\sqrt{|D|}}-2079-32\sqrt{e^{\pi\sqrt{|D|}}-2079}.
\end{equation}
But if $a \neq 1$ or $E \neq D$ then again by Lemma \ref{lem:david3} 
$$|y|=\left|j\left({b+\sqrt{E} \over 2a}\right)\right|\leq e^{\pi\sqrt{|D|-1}}+2079<e^{\pi\sqrt{|D|}}e^{-\pi/(2\sqrt{|D|})}+2079$$
a contradiction for $|D| \geq 20$ (and even $|D| \geq 15$).

Thus $\sigma={b+\sqrt{D} \over 2}$ with $b=0,1$. If $b=0$ then $y=x$. But
$$F(x,x)=-1024(x-128)(x-2048)^2$$
and neither 128 nor 2048 are singular moduli,
cf. the
complete list of rational singular moduli in Section 12.C. \cite{Cox}.
 If $b=1$ then $j(\sigma)=j(\tau+1/2)$ so $\Phi_4(x,y)=0$ for the modular transformation polynomial. Now the real roots $x>1000000$ of the resultant of $F(x,y)$ and $\Phi_4(x,y)$ with respect to $y$ are about
$$8.219997135 \cdot 10^{10},\quad 8.225266165 \cdot 10^{10}.$$
These come nowhere near the $j(\sqrt{D}/2)$ for $|D| \geq 20$, apart from $D=-64$, where they are both are dangerously close to
$$j\left({\sqrt{-64}\over2}\right)=41113158120+29071392966\sqrt{2}\approx
8.222631632 \cdot 10^{10}.$$
However the latter cannot be a zero of the resultant, because that has irreducible factors of degrees only 4,6,7.

If $-1<t<0$ we get the same lower bound for $y$ as for $t<-1$ and so we can proceed as above.

If $0<t<1/2$ 
it is even better.

If $1/2<t<1$ then we get
$$0<j(\sigma)-1728 \leq {1363\cdot265 \over e^{\pi\sqrt{|D|}}-2079}<1.$$

Now $j$ is real only on the boundary of the fundamental domain or the
 part on the imaginary axis. And $j>1728$ only on the latter. It
 follows from Lemma \ref{lem:david4} 
 that
$$0 < {\sqrt{|E|} \over 2a}-1 < {1 \over 150}\sqrt{{1363\cdot265 \over e^{\pi\sqrt{|D|}}-2079}}< {1 \over 150}.$$
Thus
$$0<|E|-4a^2 = 4a^2\left({\sqrt{|E|} \over 2a}-1\right)\left({\sqrt{|E|} \over 2a}+1\right)<{4|D| \over 3}{1 \over 150}\sqrt{{1363\cdot265 \over e^{\pi\sqrt{|D|}}-2079}}\left(2+{1\over150}\right)<1$$
for $|D| \geq 20$, which contradicts the Fundamental Theorem of Transcendence.

If $1<t<2$ it is much the same.

And if $t>2$ it is the same as $0<t<1/2$.

This completes the case of even $D$, so we next assume $D$ is odd, now
$|D| \geq 15$. Then by conjugating we can assume
$\tau=(1+\sqrt{D})/2$. By Lemma \ref{lem:david3} 
$$R(t)=j(\tau)=-|j(\tau)|\leq -e^{\pi\sqrt{|D|}}+2079\leq -e^{\pi\sqrt{15}}+2079<-190000.$$
Now $t$ is non-real and Lemma \ref{lem:david2} 
 gives six more subcases, but up to complex conjugation only three.

If $t=1/2+iv$ with $v>\sqrt{3}/2$ then
$$|y| \geq e^{\pi\sqrt{|D|}}-2079$$
and we can argue as in (\ref{eq:old13}), even for $|D| \geq 15$ (now there are no real roots $x<-190000$ of the resultant).

If $t=u+i\sqrt{1-u^2}$ with $u>1/2$ and $\sqrt{1-u^2}>0$ then
$$-2<-{2593\cdot128 \over e^{\pi\sqrt{|D|}}-2079}<j(\sigma)-1728<0.$$
Now $\sigma$ must be on the boundary but not its vertical part:
$\sigma=e^{i\theta}$ with $\pi/3 \leq \theta \leq 2\pi/3$ but
$\theta\neq \pi/2$. Thus Lemma
\ref{lem:david5} 
gives
$$0 < \left|\theta-{\pi\over2}\right| \leq {1 \over 100}\sqrt{{2593\cdot128 \over e^{\pi\sqrt{|D|}}-2079}}.$$
So
$$0<\left|{\sqrt{|E|} \over 2a}-1\right|=|\sin\theta-1|\leq{1 \over 2} \left|\theta-{\pi\over2}\right|^2,$$
much better than before.

If $t=1+u+i\sqrt{1-u^2}$ with $u<-1/2$ and $\sqrt{1-u^2}>0$ then it is pretty much the same as $t=1/2+iv$.

This finishes the largish discriminants.

If $|D|<20$ is even then $-D=4,8,12,16$ and
$$x=1728,8000,54000,287496.$$
The first implies $y=21952/9$ not a singular modulus because not an algebraic integer. The second implies $y=10976$ not a singular modulus, or $49y^2-358528y+481890304=0$ also not an algebraic integer. The third and fourth give no algebraic integers.

If $|D|<15$ is odd then $-D=3,7,11$ and
$$x=0,-3375,-32768.$$
The first implies $y=2048/3$ not an algebraic integer. The second and third give no algebraic integers.

This completes the proof of Theorem 1.\qedhere




\section{$(2,*)$ is torsion and $E_t$ has CM $\{1,3\}$}
\label{sec:ab_general}

In this section we consider the pair $\{ 1,3\}$. We prove the following
\begin{theorem}
\label{thm:abgeneral}
  Suppose $\legendre\in\IQbar\ssm\{0,1\}$ such that $E_\legendre$ has complex
  multiplication.
If  $(2,*) \in E_\legendre(\IQbar)$
has finite order $n$, then
\begin{equation*}
n = 2 \quad\text{and}\quad \legendre = 2
\end{equation*}
or
\begin{equation*}
n = 6 \quad\text{and}\quad
\legendre^2-16\legendre+16=0.
\end{equation*}
Conversely, if $n=2$ or $n=6$ then $E_\legendre$ has complex
multiplication if $\legendre$ is a root of  the corresponding polynomial.
\end{theorem}
Our method uses a variation on
 Parish's argument
\cite{Parish} to bound the number of rational torsion points on elliptic curves
with complex multiplication.

Throughout this section we suppose that $E_\legendre$ has complex
multiplication by an order in an imaginary quadratic number
field $K$.
The discriminant of $K$ will be denoted by $\Delta_K < 0$.
If will also be convenient to
write $j\in \IQ(\legendre)$ for the $j$-invariant of $E_\legendre$,
it is given by (\ref{eq:jlambda}).

Let us first treat the case where $E_\legendre$ has additional
automorphisms.
\begin{lemma}
\label{lem:extraauto}
  If $t\not=2$ and $E_t$ has $j$-invariant $1728$ or $0$, then
$(2,*)$ has infinite order.
\end{lemma}
\begin{proof}
For $t\not=2$, the order of $(2,*)$ does not divide $2$ and  we have
\begin{equation*}
  [2](2,*) = \left(-\frac 18 \frac{(t-4)^2}{t-2},*\right).
\end{equation*}

  The $j$-invariant of $E_t$ is $1728$ if and only if $t\in\{-1,1/2,2\}$.
It is $0$ if and only if $t=\zeta^{\pm 1}$ with $\zeta$ a primitive
sixth root of unity.
Thus the abscissa of $[2](2,*)$ is
\begin{equation*}
 \left\{
  \begin{array}{ll}
    \frac{25}{24} &: \text{if $t=-1$,}\medskip\\
    \frac{49}{48} &: \text{if $t=1/2$, and}\medskip\\
    \frac{15}{16}\pm \frac{\sqrt{-3}}{48} &: \text{if $t=\zeta^{\pm 1}$.}
  \end{array}\right.
\end{equation*}
We observe that the abscissa is never
integral above $2$ and above $3$.
By Theorem VII.3.4(a) \cite{Silverman:AEC}
the point $[2](2,*)$ and hence $(2,*)$ has infinite order
if $t=-1$ or $t=\zeta^{\pm 1}$.
If $t=1/2$, the model determining $E_t$ is not integral at $2$.
But it is integral  in the coordinates
$x'=4x$ and $y'=8y$. There the abscissa of
$[2](2,*)$  is $49/12$ and therefore still non-integral above $2$ and
above $3$. As before we conclude that $(2,*)$ has infinite order.
\end{proof}

We also use $E_{\overline \legendre}$ to
denote the elliptic curve defined over any field $k$ of characteristic
unequal to $2$ given by the Legendre parameter
$\overline \legendre \in k\ssm \{0,1\}$.

A place $v$ of a number field $F$ is an extension from $\IQ$ to $F$ of either the
archimedean absolute value or a $p$-adic absolute value for some prime  $p$.
In the latter case we write $v\mid p$ and let $k_v$ denote the residue
field of $v$. We can and will  identify places $v\mid p$ with prime ideals in
the ring of integers of $F$ containing $p$.

In the lemma below let $(\frac{\cdot}{\cdot})$ denote the Kronecker
symbol.

\begin{lemma}
\label{lem:congruenceprep}
Let $p\ge 3$ be a prime
 with $(\frac{\Delta_K}{p})=-1$
and suppose
$v\mid p$ is a place of $K(t)$ with
 $|\legendre|_v = |\legendre-1|_v=|j|_v = |j-1728|_v = 1$.
We write $\overline\legendre\in k_v\ssm\{0,1\}$ for the reduction
of $\legendre$ modulo $v$.
Then $E_{\overline\legendre}$ is a supersingular elliptic curve and
 $\overline\legendre\in\IF_{p^2}$. Moreover, there exists
 $\epsilon\in \{\pm 1\}$ such that
 \begin{equation*}
    E_{\overline\legendre}(\IF_{p^{2m}}) \cong (\IZ/|(\epsilon p)^m-1|\IZ)^2
 \end{equation*}
for all integers $m\ge 1$.
\end{lemma}
\begin{proof}
First we claim that if $\overline E$ is a supersingular elliptic curve defined over
$\IF_{p^2}$ with $p\ge 3$ whose $j$-invariant is not among
$\{0,1728\}$, then there exists $\epsilon \in \{\pm 1\}$ such that
$\overline E(\IF_{p^{2m}}) \cong (\IZ/|(\epsilon p)^m-1|\IZ)^2$
for all $m\ge 1$.

 Let  $[n]$ denote
multiplication by $n$ endomorphism of $\overline E$ for any $n\in\IZ$.
As $\overline E$ is supersingular, $[p]$  is a purely inseparable
isogeny of height $2$.
By Proposition 5.4, Chapter 13 \cite{Husemoller} there is an
automorphism $u$ of $\overline E$ such that $u[p] =
{\rm Fr}_{p^{2}}$ is the
Frobenius endomorphism of degree $p^{2}$ of
$\overline E$.
In other words, ${\rm Fr}_{p^{2}}$ raises the affine coordinates to the power
$p^{2}$.
By the hypothesis on the   $j$-invariant of
$\overline E$ Theorem 10.1, Chapter III \cite{Silverman:AEC} implies  that
the only automorphisms are $\pm 1$.
Therefore, ${\rm Fr}_{p^{2}}  = [\epsilon p]$ for some
$\epsilon \in \{\pm 1\}$.
Our claim follows since
$\overline E(\IF_{p^{2m}})$ is the kernel of the separable
isogeny ${\rm Fr}_{p^{2}}^m-[1] = [(\epsilon p)^m-1]$.

We apply this claim to $E_{\overline\legendre}$, which is a
well-defined elliptic curve since $|\legendre|_v=|\legendre-1|_v=1$
and $p\not=2$.
Recall that the endomorphism ring of $E_\legendre$ is an order in the
imaginary quadratic field $K$.
The prime $p$ is inert in $K$, so
$p$ generates a prime ideal in the ring of integers of $K$.
By Theorem 12, Chapter 13 \cite{Lang:elliptic}, the
elliptic curve $E_{\overline\legendre}$ is supersingular.
The $j$-invariant $\overline j$ of $E_{\overline\legendre}$
does not lie in  $\{0,1728\}$ by hypothesis. The lemma will follow
once  we can establish $\overline t\in \IF_{p^2}$.

The fact that the Legendre parameter of a supersingular elliptic curve
in characteristic $p\ge 3$ lies in $\IF_{p^2}$ is well-known, see
Dwork's note just after the proof of his Lemma
8.7 \cite{Dwork:padiccycles}.
We give a self contained proof in our situation using the claim above.
Indeed, by Deuring's Theorem,  $\overline j$  lies in
$\IF_{p^2}$, cf. Theorem V.3.1 \cite{Silverman:AEC}.
Thus
$y^2+xy = x^3 - 36x /(\overline j-1728) - 1/(\overline j-1728)$
determines an elliptic curve $\overline E$ over $\IF_{p^2}$ with
$j$-invariant  $\overline j$. By our
claim above in the case $m=1$ and since $\epsilon p - 1\equiv 0 \imod 2$ we see that
all three points of order two of $\overline E$ are defined over
$\IF_{p^2}$.
So $\overline\legendre \in\IF_{p^2}$.
\end{proof}


\begin{lemma}
\label{lem:congruence}
Let  $p$ and $v$ be as in the previous lemma.
Suppose  that $j\not\in \{0,1728\}$ and that
$E_\legendre$
contains a point of finite order $n\ge 2$ with $p\nmid n$
whose abscissa is a  rational number with denominator coprime to $p$.
Then
\begin{equation*}
p^2\equiv 1  \imod n.
\end{equation*}
\end{lemma}
\begin{proof}
  We keep the notation from the proof of the previous lemma
and suppose the point in question has abscissa $\xi\in \IQ$.

By hypothesis, the reduction $\overline\xi$ of $\xi$ modulo $v$ is
well-defined.
We obtain a point on
$E_{\overline\legendre}$ whose
ordinate
 $\overline\eta$ satisfies
$\overline \eta^2 = \overline \xi
( \overline \xi-1)( \overline \xi-\overline \legendre)$.
By the previous lemma we know that $\overline\legendre\in\IF_{p^2}$,
so  $\overline\eta\in\IF_{p^{4}}$.

Now the order of the reduced point
 $(\overline\xi,\overline\eta)$ equals $n$, the order
of $(\xi,*)$, because $n$ is coprime to the residue characteristic.
The previous lemma applied to $m=2$
 yields  $E_{\overline\legendre}(\IF_{p^4}) \cong
(\IZ/(p^2-1)\IZ)^2$ and so
$n \mid p^2-1$, as desired.
\end{proof}

For any integer $n \ge 1$ we set
\begin{alignat*}1
  \cN(n) =
 \{a \in(\IZ/n\IZ)^\times:\,\, &a^2 =  1 \}.
\end{alignat*}
In the next lemma we use Euler's totient
function $\varphi$.

\begin{lemma}
\label{prop:boundN}
Suppose that $j\not\in \{0,1728\}$.
If $E_\legendre$ contains a point of finite order $n\ge 2$ whose
abscissa is a rational number, then
\begin{alignat}1
\label{eq:triplesum}
\frac{1}{2}\varphi(n)
&<  \#\cN(n).
\end{alignat}
\end{lemma}
\begin{proof}
We fix a constant $C \ge 3n$, which may also depend on $t$ and the
point of finite order, with the following property.
If $p$ is a prime with $p\ge C$ and $v$ is any place of $K(t)$ above $p$,
then $|\legendre|_v=|\legendre-1|_v=|j|_v = |j-1728|_v = 1$ and $p$ does not divide
the denominator of the abscissa from the assertion.
Any prime $p$ in 
\begin{equation*}
  \cP_1 = \left\{ p\text{ is a prime}:p\ge C \text{ and }
\left(\frac{\Delta_K}{p}\right)=-1\right\}.
\end{equation*}
 satisfies the hypothesis of Lemma
\ref{lem:congruence}.
Therefore, $p$ lies in
\begin{equation*}
  \cP_2 = \left\{p \text{ is a prime}: p^{2}\equiv 1 \imod n
\right\}.
\end{equation*}
In other words, we have
$\cP_1 \subset \cP_2$.

We will now extract a weak form of the lemma by comparing the
density
of these two sets.
Indeed, by Chebotarev's Density Theorem we have
\begin{equation*}
\# \{p\in \cP_1:\,\, p\le T\} = \frac 12 \frac{T}{\log T}
+ o\left(\frac{T}{\log T}\right).
\end{equation*}
 as $T\rightarrow \infty$.
The same theorem tells us that primes are equidistributed among the
$\varphi(n)$
residues in $(\IZ/n\IZ)^\times$.
Thus
\begin{equation*}
\# \{p\in \cP_2:\,\, p\le T\}
= \frac{\#\cN(n)}{\varphi(n)} \frac{T}{\log T} +
o\left(\frac{T}{\log T}\right)
\end{equation*}
as $T\rightarrow +\infty$. As the density of primes in $\cP_1$ is at
most the density of primes in $\cP_2$ we conclude
\begin{equation*}
\frac 12 \varphi(n) \le \#\cN(n).
\end{equation*}

It remains to show that strict inequality holds. It is this apparently
minor strengthening  that makes the verification below harmless on
current computer hardware.
We will assume $\varphi(n)/2=\#\cN(n)$ and deriving a contradiction.

Suppose there is a prime $p\ge C$ with
$\left(\frac{\Delta_K}{p}\right)=1$ with residue in  $\cN(n)$. If $p'$ is a
further prime with $p' \equiv p \imod {n\Delta_K}$, then
\begin{equation*}
\left(\frac{\Delta_K}{p'}\right)=\left(\frac{\Delta_K}{p}\right)=1
\text{ and }  p'\text{ has residue in } \cN(n)
\end{equation*}
as $\left(\frac{\Delta_K}{\cdot}\right)$ has period $|\Delta_K|$.
But the original $p$ is coprime to $n\Delta_K$ so the set
\begin{equation*}
\cP'_1 = \{p' \text{ is a prime} : p'\equiv p \imod {n\Delta_K} \}
\end{equation*}
has positive density $1/\varphi(n|\Delta_K|)$. It is disjoint from $\cP_1$, so
$\cP_1\cup \cP'_1 \subset \cP_2$  contradicts the fact that $\cP_1$
and $\cP_2$ have equal density $1/2$.

We have established that any sufficiently large prime $p$ with
$\left(\frac{\Delta_K}{p}\right)=1$ satisfies $p^2\not\equiv 1 \imod n$.
On the other hand, there exist infinitely many primes $p$ with
$p\equiv 1 \imod {n\Delta_K}$. Each such prime
satisfies $p\equiv 1 \imod{\Delta_K}$, so
$\left(\frac{\Delta_K}{p}\right)=1$, and
$p\equiv 1 \imod n$, so $p^2\equiv 1 \imod n$. This is
  a contradiction.
\end{proof}

The inequality in the previous lemma imposes a strong restriction on
$n$.
\begin{lemma}
If $n\ge 2$ is an integer with $\varphi(n)/2 < \#\cN(n)$, then
$n\mid 24$.
\end{lemma}
\begin{proof}
By the Chinese Remainder Theorem we
find $\#\cN(nn')= \#\cN(n)\#\cN(n')$ if $n$ and $n'$ are coprime
integers.
We factor $n=\ell_1^{e_1}\cdots \ell_g^{e_g}$ into pairwise distinct
primes $\ell_i$ with exponents
$e_i\ge 1$ and find
\begin{equation*}
\prod_{i=1}^g \frac{\ell_i^{e_i-1}(\ell_i-1)}{\#\cN(\ell_i^{e_i})}=
\frac{\varphi(n)}{\#\cN(n)} < 2.
\end{equation*}
Each factor on the left is at least $1$ since $\cN(\ell_i^{e_i})$ is a subset
of $(\IZ/\ell_i^{e_i}\IZ)^\times$. So if $\ell=\ell_i$ and
$e=e_i$ for some $1\le i\le g$, then
\begin{equation*}
\ell^{e-1}(\ell-1) < 2 \#\cN(\ell^e).
\end{equation*}
We now bound $\#\cN(\ell^e)$ from above to find a restriction on
$\ell^e$.

First suppose $\ell \ge 3$. The group $(\IZ/\ell^e\IZ)^\times$ is
cyclic and therefore there are at most two solutions of $a^2=
1$ in $(\IZ/{\ell^e}\IZ)^\times$. Thus $\ell^{e-1}(\ell-1) < 4$ which implies
$\ell=3$. We find $\ell^{e-1}<2$ and hence
$e=1$.

Now suppose $\ell=2$ and $e\ge 2$. 
The group $(\IZ/2^e\IZ)^\times$ is
isomorphic
to $(\IZ/2\IZ)\times (\IZ/2^{e-2}\IZ)$. This leaves us  with at
most $4$ possibilities for $a\in(\IZ/2^e\IZ)^\times$ with $a^2=1$. As in the last paragraph we find
$2^{e-1} < 8$ and thus $e\le 3$.

The two previous paragraphs together imply $n\mid 2^3 \cdot 3 = 24$.
\end{proof}

Note that $\varphi(120)/2 = 16$ which equals $\#\cN(120)$. So the strict
inequality in Lemma \ref{prop:boundN} saves us from
 having to deal with a point of order $120$.

We combine  the conclusion of Lemma \ref{prop:boundN} with
the one of the previous lemma and Lemma \ref{lem:extraauto} to get the following statement. If an elliptic curve in
Legendre form with complex multiplication  and $j$-invariant not among
$\{0,1728\}$ contains a point of finite
order $n\ge 2$ whose abscissa is rational, then $n\mid 24$.
It would be interesting to have an explicit description of all
  $\legendre$ such that $E_\legendre$ has complex
multiplication and contains a point of finite order $>2$
with  rational abscissa.

To prove Theorem \ref{thm:abgeneral} we handle the case where the
abscissa is $2$.

\begin{proof}[Proof of Theorem \ref{thm:abgeneral}]
Say $E_\legendre$ has complex multiplication and $(2,*)$ has finite
order $n\ge 2$.
Observe that $E_2$ has $j$-invariant $1728$ and complex multiplication
by $\IZ[\sqrt{-1}]$, moreover $(2,0)$ has order $2$ on this curve.
This corresponds to $t=2$ and $n=2$.
So let us assume $t\not=2$,
by Lemma \ref{lem:extraauto} we may assume that the $j$-invariant of $E_\legendre$
is not among $\{0,1728\}$.

We have $B_n(2,t) = 0$
for the denominator of $B_n(x,t)$ of abscissa of the
multiplication by $n$ map for some  $n\mid 24$.

Using {\tt pari/gp} we compute the polynomials $B_n$ for $n$ up to
$24$ and substitute $2$ for $x$. We thus obtain polynomials in $t$,
one for each $n$, with rational coefficients. We proceed by factoring
these polynomials over the rationals. Thus we obtain potential minimal
polynomials with integer coefficients of candidate legendre parameters. We will now eliminate
most of these polynomials using classical properties of singular
moduli.

The $j$-invariant of $E_{\legendre}$ is an algebraic integer.
If $v$ is a finite place $\IQ(\legendre)$ with $v\nmid 2$, then using
the ultrametric triangle inequality together with (\ref{eq:jlambda}) we get
$|\legendre|_v=1$. As $\legendre-1$  determines the
same elliptic curve as $E_\legendre$ up to isomorphism we also
get $|\legendre-1|_v=1$. Thus we  eliminate all minimal polynomials
constructed before where the leading or constant coefficient is not a
power of $2$.

Next we take those polynomials that survive and compute their
resultant with $j \legendre ^2(\legendre-1)^2 - 2^8
(\legendre^2-\legendre+1)^3$ taken as a polynomial in $\legendre$ and
coefficients in $\IZ[j]$.
We thus eliminate $\legendre$ and obtain candidate minimal polynomials
for  the $j$-invariant after factorizing in the polynomial ring over
the integers. From these factors we
remove those that are non-monic. So the $j$-invariant $j$ of
$E_\legendre$ is a root of one of the irreducible polynomials
\begin{alignat}1
\nonumber
&J - 1728, \\
\nonumber
&J - 54000,  \\
\label{eq:finalfour3}
&J^2 - 1230272J + 1783774976
\end{alignat}
where we have omitted  a fourth irreducible polynomial (of degree $16$ which has a
coefficient greater than $10^{73}$)  whose
 reduction  modulo $5$ splits into distinct irreducible factors over $\IF_5$ as follows
\begin{equation*}
(J^2+3J+4)(J^3+4J^2+4J+2)(J^{11}+3J^{10}+2J^9+3J^8+2J^7+3J^6+J^4+3J^3+2J^2+4J+3).
\end{equation*}
 If $j$ is a root of this fourth polynomial, then $5$ splits into a product
of three prime ideals in the ring of integers in $\IQ(j)$. The residue
degrees are $2,3,$ and $11$. The
extension  $K(j)/\IQ$ is Galois by 
 Lemma 9.3 and Theorem 11.1 \cite{Cox}. Thus 
 the quotient of two residue degrees in $\IQ(j)$ above $5$
is $1/2,1,$ or $2$.
So we can exclude this fourth polynomial.

We have already treated the case $j=1728$ at the beginning of this
proof.
To eliminate the  two remaining
 polynomials  we will proceed as follows.

In fact, $54000$ is the $j$-invariant of the elliptic curve
$E_\legendre$ with
\begin{equation*}
(\legendre^2-16\legendre+16)
(\legendre^2+14\legendre+1)
(16\legendre^2-16\legendre+1)
=0.
\end{equation*}
One checks readily that the point $(2,*)$ has finite order $6$ if
$\legendre^2-16\legendre+16=0$,
this is consistent with the statement of our theorem .

If $\legendre^2+14\legendre + 1=0$ then we claim that $(2,*)$ does not
have finite order. Indeed, in this case
$\legendre = -7 \pm 4\sqrt{3}$ and we compute
\begin{equation*}
[2](2,*) = \left(\frac{155}{88}\mp \frac{29}{66}\sqrt{3},*\right).
\end{equation*}
The abscissa of $[2](2,*)$ is not integral above $2$ and above
$3$. So $[2](2,*)$ has infinite order.

If $16\legendre^2-16\legendre+1=0$, then we can argument similarly, in
this case $\legendre=1/2\pm \sqrt{3}/4$ and
$[2](2,*)= (185/176\pm 31 \sqrt{3}/1056,*)$ is not integral at a place
above $11$
and above $3$ and therefore of infinite order.

We complete the proof by ruling out that a root $j$ of
 (\ref{eq:finalfour3})
is the $j$-invariant of an elliptic curve with complex multiplication.
Let us consider the elliptic curve
$E$ given by
$y^2+xy = x^3 - 36x /(j-1728) - 1/( j-1728)$; its $j$-invariant is
 just $j$.
 The polynomial (\ref{eq:finalfour3})
  splits in
$\IQ(\sqrt{5})$ and has a root modulo $11$ represented by $2$.
The curve $E$ has good reduction at the place of $\IQ(\sqrt 5)$ above
 $11$ that corresponds to $2$.
The reduced curve is defined over  $\IF_{11}$ and the trace of its Frobenius
is $4$. So the reduced curve is ordinary. Its  endomorphismus algebra
is the imaginary quadratic field $\IQ(\sqrt{-7})$
as $4^2-4\cdot 11=-28$.
We repeat the same compution with $p=19$;
 modulo $19$ we find that (\ref{eq:finalfour3}) has two distinct roots, one is represented by
  $-2$. This time the reduced elliptic
curve over $\IF_{19}$
has trace of Frobenius $-4$.
 Its endormorphism algebra is $\IQ(\sqrt{-15})$.
Recall that reducing an elliptic curve induces an injection of the
 corresponding
 endomorphism rings. As $\IQ(\sqrt{-7})\cap\IQ(\sqrt{-15})=\IQ$ we
conclude that any root of (\ref{eq:finalfour3})  determines an elliptic curve
without complex multiplication.
\end{proof}


\section{$E_t$ has CM and $t$ is a root of unity $\{1,4\}$}
\label{sec:ac0}

In this section we consider the pair $\{ 1,4\}$, and show the following.
\begin{theorem}
\label{thm:ac0}
  Suppose $\legendre\in\IQbar\ssm\{-1,0,1\}$ such that $E_\legendre$ has complex
  multiplication. If $t$ is a root of unity, then $t=-1$ or $t= e^{\pm \pi i/3}$.
Conversely, if $\legendre=-1$ or $t= e^{\pm \pi i/3}$, then $E_\legendre$ has complex
multiplication.
\end{theorem}

The second statement is easy. Indeed,
the elliptic curve $E_{-1}$ has complex multiplication by
the Gaussian integers. If $t$ has order $6$, then $E_t$ has $j$-invariant
$0$ by (\ref{eq:jlambda}) and thus has complex multiplication by the
ring of Eisenstein integers.

So let us assume that the
order $m$ of $t$  is not in $\{1,2,6\}$. We suppose that $E_t$ has
complex multiplication and eventually derive a contradiction.
The endomorphism ring of $E_t$ is
 an order in an imaginary quadratic number field $K$.
By (\ref{eq:jlambda}), the $j$-invariant of $E_t$ is $j = J(t) \in \IQbar$ where $J$ is the
 rational function
$J(x) =   2^8 (x^2-x+1)^3x^{-2} (1-x)^{-2}$.


\begin{lemma}
  The function $f:(0,1)\rightarrow\IC$ given by
$f(\theta) = J(e^{2\pi i\theta})$
is real-valued, satisfies $f(\theta)=f(1-\theta)$, and
the  fiber $f^{-1}(f(\theta))$ contains $2$ elements if $\theta\not=1/2$.
\end{lemma}
\begin{proof}
  The first $2$ claims follow from
$\overline{f(\theta)} = J(\overline{e^{2\pi i \theta}}) = J(e^{-2\pi i \theta}) = J(e^{2\pi i \theta})=f(\theta)$ where we used
 $J(x)=J(x^{-1})$.
It is well-known that
if  $x\in \IC\ssm\{0,1\}$, then the fiber of $J$ through $x$ is
$\{x,1/x,1-x,1/(1-x),x/(x-1),(x-1)/x\}$.
Say $x$ lies on the unit circle and $1-x$ does not. Then the fiber of
$J$ restricted to the unit circle containing $x$ contains only $x$ and $1/x$. These are distinct if
$x\not=\pm 1$.
If $x$ and $1-x$ both lie on the unit circle, then $x=e^{\pm \pi i/3}$.
Here  too the fiber of $J$ containing $x$ consists only of $x$ and $1/x$.
\end{proof}

Each conjugate of $j$ over $\IQ$ is of the form
$f(k/m)$ for some integer $1\le k\le m-1$  that is coprime to $m$.
So $j$ is totally real by the previous lemma.
Since $m\not=2k$ we also find that $j$ has half as many conjugates as
$t$. Thus $[\IQ(j):\IQ]= [\IQ(t):\IQ]/2$ and so
$[\IQ(t):\IQ(j)]=2$.

\begin{lemma}
  The order $m$ divides $240$.
\end{lemma}
\begin{proof}
As we have seen in the previous section,  the extension $K(j)/\IQ$ is
Galois.
But
 Lemma 9.3 and Theorem 11.1 \cite{Cox} even tell us that its Galois
 group is
 isomorphic to ${\rm Gal}(K(j)/K)\rtimes \IZ/2\IZ$
where the non-trivial element in $\IZ/2\IZ$ acts by inversion on
${\rm Gal}(K(j)/K)$.

We split-up into two cases.

  In the first case we assume that $K\subset \IQ(t)$. Then $K(t) / \IQ$ is an
  abelian extension  whose  Galois group is isomorphic to the unit group
  $(\IZ/m\IZ)^\times$.
In particular, $K(j)/\IQ$ is abelian. The action of $\IZ/2\IZ$
  described above  implies that
$\ex{{\rm Gal}(K(j)/K)}$ divides $2$
where
$\ex{G}$ denotes the exponent of an abelian group $G$.

It seems that we are dangerously close to the notorious and unsolved problem of
explicitly
determining imaginary quadratic number fields whose class group have
exponent $2$.
Of such, 65 are known classically.  Weinberger \cite{Weinberger} proved that
there is at most additional example. Luckily, we can exploit the
structure of
${\rm Gal}(\IQ(t)/\IQ)$ to bypass these issues of ineffectivity.

We recall $[K(t):\IQ(j)]=[\IQ(t):\IQ(j)]=2$ and  $[K(j):\IQ(j)]=2$ since
$j$ is totally real. So $K(j)=K(t)$ and therefore
 $\ex{{\rm Gal}(K(t)/K)}$ divides $2$.
The Galois group of $K(t)/K$ is isomorphic to an  index $2$
subgroup
of $(\IZ/m\IZ)^\times$. Hence $\ex{(\IZ/m\IZ)^\times}$ divides 4.

The second case is when $K\not\subset \IQ(t)$, but the argument is similar. By Galois Theory
we find that $K(t)/\IQ$ is Galois with
group ${\rm Gal}(\IQ(t)/\IQ) \times {\rm Gal}(K/\IQ)\cong
(\IZ/m\IZ)^\times\times\IZ/2\IZ$.
Thus ${\rm Gal}(K(t)/K)$ is isomorphic to $(\IZ/m\IZ)^\times$.
As in the first case, the extension $K(j)/\IQ$ is abelian and we
find that ${\rm Gal}(K(j)/K)$ has exponent dividing $2$.
We recall that $\IQ(t)/\IQ(j)$ is quadratic, and  in particular
Galois. Thus $K(t)/K(j)$ has the same degree as
$\IQ(t)/\IQ(t)\cap K(j)$, which is at most $2$.
Again we find that
 $\ex{{\rm Gal}(K(t)/K)}=\ex{(\IZ/m\IZ)^\times}$ divides $4$.

In both cases we have
\begin{equation}
\label{eq:exp24}
 \ex{(\IZ/m\IZ)^\times} \mid 4
\end{equation}
which implies  the assertion as follows.
Let us abbreviate $e\ge 0$ for the largest
power of $2$ dividing $m$ and
write
$m=2^e \ell_1^{e_1} \cdots \ell_g^{e_g}$ with
$\ell_1,\ldots,\ell_g$ odd, distinct
prime divisors of $m$. By the Chinese Remainder Theorem we have
\begin{equation*}
  (\IZ/m\IZ)^\times = (\IZ/2^e\IZ)^\times \times
\prod_{i} (\IZ/\ell_i^{e_i}\IZ)^\times.
\end{equation*}
As in Section \ref{sec:ab_general}, each unit group $(\IZ/\ell_i^{e_i}\IZ)^\times$ is cycle of order
$\ell_i^{e_i-1}(\ell_i-1)$ since each $\ell_i$ is odd. The remaining factor satisfies
\begin{equation*}
  (\IZ/2^e\IZ)^\times \cong
\left\{
\begin{array}{ll}
  0 &: \text{if $e\le 1$,}\\
 \IZ/2\IZ\times \IZ/2^{e-2}\IZ &: \text{if $e\ge 2$.}
\end{array}\right.
\end{equation*}

Now we use (\ref{eq:exp24}) to restrict $m$ as follows. For any $i$ we
have $(\ell_i-1) \mid 4$.
So the only possible odd prime divisors of $m$ are
$3$ and $5$. Their respective squares cannot divide $m$.
Finally, we must have $e-2 \le 2$. So $e\le
4$ which implies $m\mid 2^4\cdot 3\cdot 5=240$.
\end{proof}

\begin{proof}[Proof of Theorem \ref{thm:ac0}]
The divisors of $240$ are
\begin{equation}
\label{eq:mlist}
1, 2, 3, 4, 5, 6, 8, 10, 12, 15, 16, 20, 24, 30, 40, 48, 60, 80, 120, 240.
\end{equation}


Let us suppose that $m$ is a power of an odd prime $p$.
Then $|1-t|_v < 1$ for any place $v$ of $\IQ(t)$ above $p$.
From this we easily deduce  $|j|_p > 1$ using (\ref{eq:jlambda}). This
contradicts  the well-known
fact that $j$ is an algebraic integer. 
This argument eliminates $m=3$ and $m=5$ but fails to cover other values.

What happens if $m=4$? Then $t^2=-1$ and $J(t) = 2^7$. Again, this is
not a singular moduli. 

We will now  eliminate $m=8$ using a method that  works for the
other divisors as well.

Suppose $t$ has precise order $8$.
 We start out by picking the auxiliary prime $p=17$. Since $p\equiv 1
 \imod 8$ it splits completely in the cyclotomic field $\IQ(t)$.
The integers $2$ and $8$ represent elements $\overline 2$ and
 $t_2=\overline 8$ in $\IF_{17}$
whose multiplicative order are both $8$. So there are two places
 $v_{1,2}\mid 17$ such that $t$ modulo $v_{1}$ is $\overline 2$
and $t$ modulo $v_2$ is $\overline 8$.
We  use these residues to construct two possible reductions
of $E_t$ in characteristic $17$ as Legendre elliptic curves.
The following table lists
the Legendre equations, the trace of Frobenius, and
the discriminant of the imaginary quadratic number field generated by
a root of the said characteristic polynomial.
\renewcommand{\arraystretch}{1.2}
\begin{center}
  \begin{tabular}{c|c|c}
Legendre curve over $\IF_{17}$ & trace of Frobenius & field
discriminant \\
\hline
    $y^2=x(x-\overline 1)(x-\overline {2})$ & $2$ & $-4$ \\
    $y^2=x(x-\overline 1)(x-\overline {8})$ & $-6$ & $-8$\\
 \end{tabular}
\end{center}
Observe that none of the reduced curves is supersingular as the trace
of Frobenius is never divisible by $17$. Thus the endomorphism rings of the
reduced curves are  orders in an imaginary quadratic field
whose discriminant is given by the  column on the right.

The endomorphism ring of $E_t$ injects into the endomorphism ring of
any of its reductions. As we are assuming that $E_t$ has complex
multiplication,  the table above leads to a contradiction.

For each remaining $m$ from (\ref{eq:mlist}) we provide an
auxiliary prime $p$
satisfying $p\equiv 1  \imod m$, two elements $t_1,t_2\in\IF_p$ of
multiplicative order $m$
that serve as the Legendre parameter for ordinary elliptic curves,
the trace of Frobenius $a_{1,2}$ for each reduction, and the discriminants of
corresponding endomorphism algebras $\Delta_{1,2}$.

\renewcommand{\arraystretch}{1.2}
\begin{longtable}{c|c|c|c|c|c|c|c}
$m$ & $p$ & $t_1$ & $t_2$ & $a_1$ & $a_2$ & $\Delta_1$ & $\Delta_2$ \\
\hline
\endfirsthead
$m$ & $p$ & $t_1$ & $t_2$ & $a_1$ & $a_2$ & $\Delta_1$ & $\Delta_2$ \\
\hline
\endhead
10  &  11  &  $\overline 2$  &  $\overline{7}$  &  $0$  &  $-4$  &  $-11$  &  $-7$\\
12  &  13  &  $\overline 2$  &  $\overline{6}$  &  $6$  &  $2$   &  $-4$   &  $-3$\\
15  &  31  &  $\overline 7$  &  $\overline{14}$ &  $8$  &  $0$   &  $-15$  &  $-31$\\
16  &  17  &  $\overline 3$  &  $\overline{5}$  &  $-6$ &  $2$   &  $-8$   &  $-4$\\
20  &  41  &  $\overline 2$  &  $\overline{5}$  &  $10$ &  $-6$  &  $-4$   &  $-8$\\
24  &  73  &  $\overline 7$  &  $\overline{17}$ &  $2$  &  $-6$  &  $-8$   &  $-4$\\
30  &  31  &  $\overline 3$  &  $\overline{12}$ &  $-4$ &  $0 $  &  $-3$   &  $-31$\\
40  &  41  &  $\overline 6$  &  $\overline{11}$ &  $2$  &  $-6$  &  $-40$  &  $-8$\\
48  &  97  &  $\overline 2$  &  $\overline{3}$  &  $18$ &  $-14$ &  $-4$   &  $-3$\\
60  &  61  &  $\overline 2$  &  $\overline{6}$  &  $-10$&  $-2$  &  $-4$   &  $-15$\\
80  &  241 &  $\overline {17}$ &$\overline{21}$ &  $18$ &  $-22$ &  $-40$  &  $-120$\\
120 &  241 &  $\overline 3 $ &  $\overline{12}$ &  $-14$&  $-22$ &  $-3$   &  $-120$\\
240 &  241 &  $\overline 7$  &  $\overline{13}$ &  $-30$&  $26$  &  $-4$   &  $-8$
%
\end{longtable}

For given $m$ in the table, the two listed field discriminants
are different. By the same argument as above this means that $E_t$ does not have complex multiplication
if $t$ has order $m$.
\end{proof}


\section{$(2,*)$ is torsion and $E_{-t}$ has CM $\{2,3\}$}
\label{sec:bg}
We consider the pair $\{2,3\}$ and prove the following.
\begin{theorem}
\label{thm:bg}
  Suppose $\legendre\in\IQbar\ssm\{0,1\}$ such that $E_{-\legendre}$ has complex
  multiplication.
If  $(2,*) \in E_\legendre(\IQbar)$
has finite order $n\ge 2$, then $n\le 10^{10^{19}}$ and the discriminant of the ring of
endomorphisms of $E_{-\legendre}$ is strictly less than $2\cdot 10^{32}$ in modulus. 
\end{theorem}
The proof splits into two parts.

In this section we need to work with the absolute logarithmic Weil height $h(x)$
of an algebraic number $x$. We refer to Chapter 1 \cite{BG} for the
definition and basic properties. 

\subsection{Bounding the height of $t$ from above}
\label{sec:bg_part1}

\begin{lemma}\label{res}
 Let $A,B$ be polynomials in $\IZ[X]$ with $\max\{\deg A,\deg B\}=d\geq1$, and suppose there exist $P_0,Q_0,P_\infty,Q_\infty$ in $\IZ[X]$ with degrees at most $d-1$ and $r \neq 0$ in $\IZ$ such that
$$P_0(X)A(X)+Q_0(X)B(X)=r,~~P_\infty(X)A(X)+Q_\infty(X)B(X)=rX^{2d-1}.$$
Then for any algebraic $x$ with $B(x) \neq 0$ we have
$$h\left({A(x) \over B(x)}\right) \geq dh(x)-\log C$$
where
$$C=\max\{L(P_0)+L(Q_0),L(P_\infty)+L(Q_\infty)\}$$
for the lengths.
\end{lemma}
\begin{proof}
 Let $K$ be a number field containing $x$. For any non-archimedean
place $|\cdot|$ on $K$ it is clear that
 \begin{equation}
\label{eq:bg1}
  |rx^{2d-1}|=|P_\infty(x)A(x)+Q_\infty(x)B(x)|\leq\max\{1,|x|^{d-1}\}\max\{|A(x)|,|B(x)|\}.
 \end{equation}
We also have
\begin{equation}
\label{eq:bg2}
 |r|=|P_0(x)A(x)+Q_0(x)B(x)|\leq\max\{1,|x|^{d-1}\}\max\{|A(x)|,|B(x)|\},
\end{equation}
and so
\begin{equation}
\label{eq:bg3}
 |r|\max\{1,|x|^{2d-1}\}\leq\max\{1,|x|^{d-1}\}\max\{|A(x)|,|B(x)|\}
\end{equation}
with the obvious cancellation. But for an archimedean place $|\cdot|$
of $K$ we get
an extra $L(P_\infty)+L(Q_\infty) \leq C$ on the far right of
(\ref{eq:bg1}), and an extra $L(P_0)+L(Q_0) \leq C$ on the far right
of (\ref{eq:bg2}). Thus an extra $C$ on the far right of
(\ref{eq:bg3}). Taking the product over all places and then the
logarithm we get the desired result.
\end{proof}

Note that the upper bound
$$h\left({A(x) \over B(x)}\right) \leq dh(x)+\log \max \{L(A),L(B)\}$$
is practically obvious.


Now considering $3P_t$ and using a Zimmer constant (see below) coming
from the Weierstrass model we already get $h(t)<75$.
But it is more efficient to calculate directly a Zimmer constant on the Legendre model.




This time working over ${\IZ}[t]$ not $\IZ$ we use
$$A(X,t)=X^4-2t X^2+t^2=(X^2-t)^2,~~B(X,t)=4X(X-1)(X-t)$$
(coming from the duplication formula) with
\begin{alignat*}1
P_0(X,t)&=-12X^2+(8t+8)X+4t^2-8t+4,\\
Q_0(X,t)&=3X^3+(t+1)X^2-5t X-2t^2-2t, \\
P_\infty(X,t)&=4t^2(t^2-2t+1)X^3+2t^2(4t^2+4t)X^2-12t^4X,\\
Q_\infty(X,t)&=-2t^3(t+1)X^3-5t^4X^2+(t^5+t^4)X+3t^5.
\end{alignat*}
We have
$$P_0(X,t)A(X,t)+Q_0(X,t)B(X,t)=r,~~P_\infty(X,t)A(X,t)+Q_\infty(X,t)B(X,t)=rX^{7}$$
with $r=4t^2(t-1)^2$.

For non-archimedean we get
\begin{equation}
\label{eq:bg1z}
  |rx^{7}|=|P_\infty(x,t)A(x,t)+Q_\infty(x,t)B(x,t)|\leq\max\{1,|x|^{3}\}\max\{1,|t|^{5}\}\max\{|A(x,t)|,|B(x,t)|\}.
\end{equation}
We also have
\begin{equation}
\label{eq:bg2z}
|r|=|P_0(x,t)A(x,t)+Q_0(x,t)B(x,t)|\leq\max\{1,|x|^{3}\}\max\{1,|t|^{5}\}\max\{|A(x,t)|,|B(x,t)|\},
\end{equation}
and so
\begin{equation}
\label{eq:bg3z}
  |r|\max\{1,|x|^{7}\}\leq\max\{1,|x|^{3}\}\max\{1,|t|^{5}\}\max\{|A(x,t)|,|B(x,t)|\}
\end{equation}
with the obvious cancellation. But for an archimedean place we get an extra $L(P_\infty)+L(Q_\infty) \leq 44+14=58$ on the far right of (\ref{eq:bg1z}), and an extra $L(P_0)+L(Q_0) \leq 44+14=58$ on the far right of (\ref{eq:bg2z}). Thus an extra $58$ on the far right of (\ref{eq:bg3z}). Taking the product over all places and then the logarithm we get
$$h(2P) \geq 4h(P)-5h(t)-\log 58$$
for any $P$ in $E_t(\overline{\bf Q})$, where the height is that of
the abscissa. A much better upper bound holds, and we deduce the
Zimmer-type bound
\begin{equation}
  \label{eq:zimmerbound}
|h(P)-\hat h(P)| \leq {1 \over 3}(5h(t)+\log 58)
\end{equation}
for the corresponding N\'eron-Tate height.

We now use $4P_t$. We find that this is
$A_4(t)/B_4(t)$ with $A_4$ of degree 8 and $B_4$ of degree
7 both in ${\IZ}[X]$. So $d=8$ in Lemma \ref{res}. We find
$P_0,Q_0,P_\infty,Q_\infty$ with $r=2^{33}$ and $C=192475067056128$;
these are surprisingly small because large powers of 2 can be
cancelled (the resultant of $A_4,B_4$ is $-2^{160}$). 
Thus $h(4P_t) \geq 8h(t)-\log C$. On the other this is
at most ${1 \over 3}(5h(t)+\log 58)$. Comparison gives
\begin{equation}
\label{eq:htbound}
  h(t) \leq 5.4070213804731854624.
\end{equation}

Here are the explicit expressions, if it helps at all. First
$$A_4=-X^8+160X^7-7104X^6+57344X^5-206336X^4+401408X^3-442368X^2$$$$+262144X-65536,
$$
$$B_4=288X^7-3648X^6+17408X^5-38912X^4+40960X^3-16384X^2.
$$
Then
$$P_0=9486432X^6-110727168X^5+464270080X^4-827747840X^3+540543488X^2$$$$+524288X+131072,$$
$$Q_0=32939X^7-5237482X^6+228793380X^5-1661849512X^4+5165533824X^3$$$$-8172300544X^2+6555431936X-2157324288,$$
$$P_\infty=8589934592X^7+21474836480X^6-867583393792X^5+3985729650688X^4$$$$-6322191859712X^3+3298534883328X^2,
$$
and finally
$$Q_\infty=-4697620480X^7+137438953472X^6+1340029796352X^5-13726715478016X^4$$$$+42365557407744X^3-62122406969344X^2+45079976738816X-13194139533312.$$


\subsection{Bounding the Faltings height from below}
\label{sec:bg_part2}
Let $\tau$ lie in the upper half-plane 
 and 
$q=e^{2\pi i \tau}$. Then 
\begin{equation*}
  \Delta(\tau) = (2\pi)^{12} q\prod_{n\ge 1} (1-q^n)^{24}
\end{equation*}
defines the discriminant function. It  appears
in the Faltings height of an
elliptic curve $E$  defined over a number field $K$. 
Indeed, Suppose that $E$ has complex multiplication. 
For each embedding  $\sigma:K\rightarrow\IC$ let $\tau_\sigma$ have
positive imaginary part and be the
quotient of a choice of period lattice basis vectors of the
 complex elliptic curve induced by $E$ and $\sigma$.
The stable Faltings height of $E$ is
\begin{equation}
\label{eq:Falthgt}
  \Fh{E} =  -\frac{1}{12 [K:\IQ]}\sum_{\sigma:K\rightarrow\IC}
 \log(|\Delta(\tau_\sigma)| \imag{\tau_\sigma}^6) +\frac 12 \log \pi,
\end{equation}
where $\imag{\cdot}$ denotes the imaginary part of a complex number. 
Each logarithm in (\ref{eq:Falthgt}) is invariant under the action
of ${\rm SL}_2(\IZ)$ by fractional linear transformations on
$\tau_\sigma$. 
So the   local terms are independent of the choice made before.  
The term $(\log \pi)/2$ is part of Deligne's normalization.

For example, the height of the elliptic curve $E$ with CM by the maximal order in
$\IQ(\sqrt{-3})$ is 
\begin{equation*}
  \Fh{E} = -0.7487524855033378279181555201\ldots.
\end{equation*}

Let $f$ be a natural number and $p$ a prime divisor of $f$. 
Suppose that $\chi$ is a quadratic character and
that $n$ is the largest integer with $p^n\mid f$. We 
set
\begin{equation}
\label{def:efp}
  e_f(p) = \frac{1-\chi(p)}{p-\chi(p)}\frac{1-p^{-n}}{1-p^{-1}}.
\end{equation}

We now make Colmez's Theorem and its extension by Nakkajima-Taguchi
explicit. This and an estimate of Badzyan will lead to a lower bound for the height of an
elliptic curve with complex multiplication.

\begin{lemma}
\label{lem:Fhlowerbound}
  Let $E$ be an elliptic curve with complex multiplication by 
an order with discriminant $\Delta f^2 $ where $f\in\IN$ and $\Delta<0$
is a fundamental discriminant.
\begin{enumerate}
\item [(i)] We have
  \begin{equation*}
    \Fh{E} = \frac{1}{4} \log(|\Delta|f^2) +
\frac 12 \frac{L'(\chi,1)}{L(\chi,1)} 
-\frac 12 \left(\sum_{p\mid f} e_f(p)\log p\right)
-\frac{1}{2} (\gamma+\log(2\pi))
  \end{equation*}
where the Kronecker symbol $\chi(\cdot)=\left(\frac{\Delta}{\cdot}\right)$ 
is used in (\ref{def:efp}) and
$\gamma  = 0.5772 \ldots$ is Euler's constant. 
\item[(ii)]
We have the estimates
\begin{alignat*}1
  \Fh{E} &\ge \frac{\sqrt{5}}{20}\log|\Delta|
+\frac 12 \log f 
-3\log(1+\log f)
-\gamma
-\frac{\log(2\pi)}{2} \\
&\ge \frac{\sqrt{5}}{20}\log(|\Delta|f^2) - 5.93.  
\end{alignat*}
\end{enumerate}
\end{lemma}
\begin{proof}
  Part (i) is classical for $f=1$ and follows
from Nakkajima-Taguchi's result  \cite{NT} in general.
For a  detailed argument see Lemma 4.1 \cite{habegger:wbhc}.

To prove the second part we use 
Badzyan's Theorem 1 \cite{Badzyan} which implies
\begin{equation*}
  \frac{L'(\chi,1)}{L(\chi,1)} +\gamma \ge  -\frac{1}{2}
\left(1-\frac{\sqrt{5}}{5}\right)\log |\Delta|. 
\end{equation*}
Indeed, the left-hand side is the said author's $\gamma_K$ with 
$K=\IQ(\sqrt\Delta)$. 
 Ihara's paper \cite{IharaKronecker} contains bounds
subject to the Generalized Riemann Hypothesis.

After rearranging and dividing by $2$ we find
\begin{equation*}
\frac 14 \log|\Delta|+  \frac 12 \frac{L'(\chi,1)}{L(\chi,1)} 
\ge -\frac{\gamma}{2}+\frac{\sqrt{5}}{20} \log |\Delta|.
\end{equation*}
So
\begin{equation*}
  \Fh{E} \ge \frac{\sqrt{5}}{20}\log|\Delta| + \frac 12 \log f
-\frac 12 \left(\sum_{p\mid f} e_f(p)\log p\right)
-\gamma -\frac{\log(2\pi)}{2}
\end{equation*}
by part (i). 

Below we will show 
  \begin{equation}
\label{eq:efpbound}
    \sum_{p\mid f} e_f(p)\log p \le 6\log(1+\log f). 
  \end{equation}
This inequality implies the first lower bound in  (ii) and makes
  explicit an estimate of Poonen \cite{Poonen:MRL01}.

To prove (\ref{eq:efpbound}) we may assume $f\ge 2$. 

 If $\chi(p)=0$ or $1$, then 
$e_f(p)\le 1/(p-1)$. Otherwise, 
$e_f(p)\le 2/((p+1)(1-1/p)) = 2/(p-1/p)$
from which we deduce $e_f(p)\le 2/(p-1)$ and $e_f(p)\le 3/p$. 
These two inequalities hold regardless of the value of $\chi(p)$. 

Let $x>1$ be a real number, then 
  \begin{equation*}
    \sum_{p\mid f} e_f(p)\log p \le 3 \sum_{p\le x} \frac{\log p}{p}
+2 \sum_{p\mid f,p> x} \frac{\log p}{p-1}.
  \end{equation*}
Rosser and Schoenfeld's Corollary to Theorem 6 \cite{RosserSchoenfeld62} yields
\begin{equation*}
  \sum_{p\le x} \frac{\log p}{p}< \log x.
\end{equation*}
The map $x\mapsto (\log x)/(x-1)$ is decreasing on $(1,\infty)$, so
$\sum_{p\mid f,p>x}\frac{\log p}{p-1}\le 
(\log x) \omega(f)/ (x-1)$
where $\omega(f)$ denotes the number of distinct prime divisors of $f$.
Combining the bounds for small and large primes yields
  \begin{equation*}
    \sum_{p\mid f} e_f(p)\log p \le 
(\log x)\left(3+ \frac{2\omega(f)}{x-1}\right).
  \end{equation*}
A reasonable choice for $x$ is $1+2\omega(f)/3 > 1$
which leads to 
$6\log(1+2\omega(f)/3)$ as an upper bound.

Inequality  (\ref{eq:efpbound}) now follows from  the
 elementary bound $\omega(f) \le (\log f)/(\log 2)$ which leads to 
 $\log(1+2\omega(f)/3) < \log (1+\log f)$ since $2 < 3\log 2$. 

It remains to show the second inequality in (ii). 
On taking the derivative we observe that 
\begin{equation*}
f\mapsto \left(\frac 12 - \frac{\sqrt 5}{10}\right) \log f -
3\log(1+\log f) 
\end{equation*}
takes its minimum in $[1,+\infty)$ at
$\exp((13+3\sqrt 5)/2)$. 
Thus its minimal value is greater than $-4.431$ and we conclude
\begin{equation*}
\frac 12 \log f - 3\log(1+\log f) \ge \frac{\sqrt{5}}{20} \log(f^2) - 4.431.
\end{equation*}
Part (ii) follows since $-4.431 - \gamma - \log(2\pi)/2 \ge -5.93$. 
\end{proof}


We define
\begin{equation*}
  \vartheta(\tau) = 16 e^{\pi i \tau} \prod_{n\ge 1}
\left(\frac{1+e^{2\pi i n\tau}}{1+e^{2\pi i (n-1/2)\tau}}\right)^8
\end{equation*}
and observe that $\vartheta = 1/(1-\lambda)$ where $\lambda$ is the
the modular function of level $2$ as in Section
18.6 \cite{Lang:elliptic}, cf. Remark 2 in Section 18.4 of
loc.~cit.

\begin{lemma}
\label{lem:localestimate}
If the imaginary part of $\tau$  
is at least $\sqrt{3}/2$, then
  \begin{equation*}
|\Delta(\tau) \vartheta(\tau)^{-2}| \ge 4160174. 
  \end{equation*}
\end{lemma}
\begin{proof}
  The left-hand side of the assertion is
  \begin{equation*}
2^{-8} (2\pi)^{12}
\left|\prod_{n\ge 1}
(1-e^{2\pi i n\tau})^{24}\left(\frac{1+e^{2\pi i (n-1/2)\tau}}{1+e^{2\pi i n\tau}}\right)^{16}\right| \ge
2^{4} \pi^{12}
 \prod_{n\ge 1}(1-q^n)^{24}\left(\frac{1-q^{n-1/2}}{1+q^{n}}\right)^{16}
  \end{equation*}
where  $q=e^{-\sqrt{3}\pi} \ge |e^{2\pi i \tau}|$. 
The right-hand side is greater than $4160174$.
\end{proof}

\begin{lemma}
\label{lem:legendreheight}
Let $t\in \IQbar\smallsetminus\{0,1\}$ such that $E_t$ has 
good reduction everywhere.
  Then 
  \begin{alignat*}1
    \Fh{E_t} 
&\le  -0.394 + \frac 12 \height{t}.
\end{alignat*}
\end{lemma}
\begin{proof}
The function
\begin{equation}
\label{eq:sl2invariant}
\tau\mapsto  |\Delta(\tau)|\imag{\tau}^6
\max\left\{|\vartheta(\tau)|,|1-\vartheta(\tau)|,\frac{1}{|\vartheta(\tau)|},\frac{1}{|1-\vartheta(\tau)|},\frac{|\vartheta(\tau)|}
{|\vartheta(\tau)-1|},
\frac{|1-\vartheta(\tau)|}{|\vartheta(\tau)|}\right\}^2
\end{equation}
is invariant under the action of ${\rm SL}_2(\IZ)$ on the upper
half-plane. If $\sigma:\IQ(t)\rightarrow\IC$ is an embedding, we take  $\tau_\sigma$ in the 
 fundamental domain
of the ${\rm SL}_2(\IZ)$ action on the upper-half plane 
and such that
$j(\tau_\sigma)$ is the image under $\sigma$ of the $j$-invariant of
$E_t$.
Thus
 $\imag{\tau_\sigma}\ge
\sqrt{3}/2$ and
 $\sigma(t)$ equals
$\vartheta(\tau_\sigma),1-\vartheta(\tau_\sigma),1/\vartheta(\tau_\sigma),1/(1-\vartheta(\tau_\sigma)),
\vartheta(\tau_\sigma)/(\vartheta(\tau_\sigma)-1)$, or
$1-1/\vartheta(\tau_\sigma)$. We use   
 Lemma \ref{lem:localestimate} to bound
(\ref{eq:sl2invariant}) from below by
$4160174\cdot  (\sqrt{3}/2)^6 > 1755073$.

We recall (\ref{eq:Falthgt}). 
Summing over all $\sigma:\IQ(t)\rightarrow\IC$, dividing by $12[K:\IQ]$, and
adding $\log(\pi)/2$ yields  
\begin{alignat*}1
  \Fh{E_t} \le &-\frac{1}{12}\log(1755073)
+ \frac 12 \log\pi 
\\
&+ \frac 16 \height{[t:1-t:t^{-1}:(1-t)^{-1}:t(t-1)^{-1}:(1-t)t^{-1}]}
\end{alignat*}
where we use the projective height in the end.
After multiplying by $t(t-1)$, the product formula implies that this
projective height equals  
\begin{equation*}
\height{[t^2(t-1):t(t-1)^2:t-1:t:t^2:(t-1)^2]}.
\end{equation*}
Local considerations at each place show that it is at most
$3\height{t}+\log 4$. 
The lemma follows since
$-\log(1755073)/12+\log(\pi)/2 +\log(4)/6 <
-0.394$. 
\end{proof}

We now complete the proof of Theorem \ref{thm:bg}. 
So we assume that
 $E_{-t}$ has complex multiplication 
 and that $(2,*)$ is a point of
 finite order $n$ on $E_{t}$.
We present the discriminant $\Delta f^2$ of the endormophism ring of
 $E_{-t}$ 
as in Lemma  \ref{lem:Fhlowerbound}. We observe that the inequalities
in (ii) hold
unchanged as $\height{t}=\height{-t}$.
Together with the previous lemma we obtain
$\sqrt{5}\log(|\Delta|f^2)/20 \le \height{t}/2 + 5.536$. 
We use the height bound
(\ref{eq:htbound}) 
to get
$|\Delta|f^2<  2\cdot 10^{32}$,
as desired. 

By an estimate involving the analytic class number formula,
 the number $d$ of positive definite, primitive
quadratic forms of discriminant $f^2\Delta$ up-to equivalence
satisfies 
\begin{equation*}
 d\le \frac{6}{2\pi} f|\Delta|^{1/2}(2+\log(f^2|\Delta|))
\le 1.1\cdot 10^{18},  
\end{equation*}
 cf. Hua's Theorems 12.10.1 and 
12.14.3 \cite{Hua}. Observe that Hua's estimate holds for
discriminants that are not necessarily fundamental. 
This class number is $d = [K(j):K] = [\IQ(j):\IQ]$ where $j$ is the
$j$-invariant of $E_{-t}$, see Theorem 11.1 \cite{Cox}.

The points $P_t=(2,*),2P_t,\ldots,(n-1)P_t$ yield at least $(n-1)/2$
distinct
abscissas. They all lie in the number field $\IQ(t)$  and 
 have height at most 
$(5h(-t)+\log 58)/3 \le \log H$ by (\ref{eq:zimmerbound}) and (\ref{eq:htbound}) from  Section
 \ref{sec:bg_part1}
where $H=31736$. By
Loher's estimate, cf. (1.5) \cite{LoherMasser}, we find
\begin{equation*}
  \frac{n-1}{2} \le 37 d(\log d) H^{2d} \le 10^{9.99 \cdot 10^{18}}
\end{equation*}
if $d\ge 2$; for $d=1$ then $\IQ(t)\subset\IQ(\sqrt 2)$ and
 left side is bounded by $37\cdot 2 \cdot
(\log 2)H^4$, which is much better. Since $(n-1)/2\ge
n/10$ we find 
$n\le 10^{10^{19}}$. 
\qed


\section{$(2,*)$ is torsion and $t$ is a root of unity $\{3,4\}$}
\label{sec:bc0}
We consider the pair $\{ 3,4\}$. We show the following.
\begin{theorem} \label{thm:bc0}
There is an absolute effective constant $C>0$ with the following property. If $t\ne 1$ is a root of unity such that $P=(2,*) \in E_\legendre(\IQbar)$ has finite order then the degree of $t$ over $\IQ$ is at most $C$.
\end{theorem}

We follow the strategy of \cite{MZ:torsionanomalousAJ}.

Let
\begin{eqnarray*}
\omega_1(t)=\int_{1}^{\infty} \frac{dX}{\sqrt{X(X-1)(X-t)}}, \\
\omega_2(t)=\int_{-i\infty}^{0} \frac{dX}{\sqrt{X(X-1)(X-t)}}, \\
\phi(t)=\frac{1}{2}\int_{2}^{\infty} \frac{dX}{\sqrt{X(X-1)(X-t)}}.
\end{eqnarray*}

These are analytic for $t$ with negative real part. Given such a $t$ the values $\omega_1(t),\omega_2(t)$ form a basis for a period lattice of $E_t$, and $\phi(t)$ is an elliptic logarithm of the point $P$ on $E_t$. We write $t= x+i y$ with real $x$ and $y$ and let
\begin{eqnarray}
\omega_1(t) & = & r_1(x,y)+is_1(x,y) \label{realimagparts1}\\
\omega_2(t) & = & r_2(x,y)+is_2(x,y) \label{realimagparts2}\\
\phi(t) & = & u(x,y)+iv(x,y)    \label{realimagparts3}
\end{eqnarray}
where the six functions on the right are real-valued real-analytic functions of $x$ and $y$. Let
\begin{equation}\label{Delta}
\Delta=r_1s_2-r_2s_1.
\end{equation}
Then for any $x,y$ such that $t$ is in the left half plane $|\Delta(x,y)|$ is the area of any fundamental domain of the lattice determined by $\omega_1(t)$ and $\omega_2(t)$. This area is non-zero, and so the following definitions make sense
\begin{eqnarray}
f_1(x,y)= \frac{ u(x,y) s_2(x,y) -v(x,y) r_2(x,y)}{\Delta(x,y)},\label{fandg1} \\
f_2(x,y)= \frac{-u(x,y) s_1(x,y) +v(x,y) r_1(x,y)}{\Delta(x,y)}.\label{fandg2}
\end{eqnarray}
These functions are real-analytic and satisfy
\begin{equation}\label{propertyoffandg}
\phi(t)=f_1(x,y)\omega_1(t)+f_2(x,y)\omega_2(t).
\end{equation}
In order to control the roots of unity $t$ such that $P$ is torsion on $E_t$ we shall prove an effective Bombieri-Pila style result for the function $f_1(\cos 2\pi z,\sin 2\pi z)$ for $z$ sufficiently near $1/2$. For this we need to study the continuation of $f_1$ as a function of two complex variables near the point $(-1,0)$ in $\IC^2$.

All the constants asserted to exist below, including implied constants, are absolute and can be effectively computed.

 \begin{lemma}\label{boundsonintegrals} There is a $\delta_1>0$ and an $M_1>0$ such that if  $|t+1| \le \delta_1$ then
$$
\max \{|\omega_1(t)|,|\omega_2(t)|,|\phi(t)|\} \le M_1.
$$
\end{lemma}
\begin{proof}This follows from Lemma A.2 from \cite[Page 477]{MasserZannierMathAnn}, with $\delta=1/2$ say, after checking that the absolute constant in Lemma A.1 on the previous page of \cite{MasserZannierMathAnn} can be effectively computed.
\end{proof}
\begin{lemma}
\label{continuationoff}
There is an $\varepsilon>0$ and an $M_2>0$ such that the function $f_1(\cos2\pi z,\sin 2\pi z) $ has a analytic continuation $f$ to the complex disk $|z-1/2|\le \varepsilon$ and on this disk $|f(z)|\le M_2$.
\end{lemma}
\begin{proof}
Let $w_1=x+ix'$ and $w_2=y+iy'$ be complex variables extending $x$ and $y$ (so $x'$ and $y'$ are real). For $w=(w_1,w_2)$ sufficiently close to $(-1,0)$ the functions
\begin{eqnarray*}\label{hats}
\hat{u} (w)& =&\frac{1}{2}\left(\phi(w_1+iw_2)+\overline{\phi(\overline{w_1}+i\overline{w_2})}\right) \\
\hat{v} (w) &=&\frac{1}{2}\left(-i\phi(w_1+iw_2)+\overline{-i\phi(\overline{w_1}+i\overline{w_2})}\right),
\end{eqnarray*}
where $\overline{\cdot}$ denotes complex conjugation, are analytic, and if $w_1=x,w_2=y$ are real then $\hat{u}(w)=u(x,y)$ and $\hat{v}(w)=v(x,y)$. We define complex analytic extensions $\hat{r}_1,\hat{s}_1,\hat{r}_2$ and $\hat{s_2}$ of $r_1,s_1,r_2$ of $s_2$ in a similar manner. We use the Euclidean norm on $\IC^2$. By Lemma \ref{boundsonintegrals}, \eqref{realimagparts1}, \eqref{realimagparts2} and \eqref{realimagparts3}, if $|w-(-1,0)| \le \delta_1/2$ then
$$
\max_{i=1,2}\{ |\hat{u}(w)|, |\hat{v}(w)|, |\hat{r}_i(w)|,|\hat{s}_i(w)|\} \le M_1.
$$
We can now define $\hat{\Delta}$ by putting hats on \eqref{Delta}, and then if $|w-(-1,0)|\le \delta_1/2$ we have
$$
|\hat{\Delta}(z,w)| \le 2M_1^2.
$$
Using this upper bound for $|\hat{\Delta}|$ we now find a lower bound. First, by Cauchy's inequalities, we have
\begin{equation}\label{cauchy}
\max \left\{ \left| \diff{\hat{\Delta}}{w_1}(w) \right|, \left| \diff{\hat{\Delta}}{w_2}(w) \right| \right\} \le 4M_1^2\delta_1^{-1}
\end{equation}
for $|w-(-1,0)|\le \delta_1/2$. So for such $w$ we have
\begin{equation}
|\hat{\Delta}(w)-\hat{\Delta}(-1,0)| \le 4\sqrt2 M_1^2\delta_1^{-1}|w-(-1,0)|.
\end{equation}
Hence if we take a $\delta_2<\delta_1/2$ such that $8\sqrt2 M_1^2\delta_1^{-1}\delta_2< |\hat{\Delta}(-1,0)|$ then for $w$ with $|w-(-1,0)|\le \delta_2$ we have 
\begin{equation}\label{arealowerbound}
|\hat{\Delta}(w)| \ge \frac{|\hat{\Delta}(-1,0)|}{2}.
\end{equation}
With the lower bound out of the way, we can quickly prove the lemma. First, using the extensions of the real and imaginary parts, we can define a complex-analytic extension $\hat{f_1}$ of $f_1$ on the set of $w$ in $\IC^2$ such that $|w-(-1,0)| \le \delta_2$. By Lemma \ref{boundsonintegrals}, the definition of $\hat{f_1}$ and \eqref{arealowerbound} we have
\[
|\hat{f_1}(w)| \le 4 M_1^2 |\hat{\Delta}(-1,0)|^{-1}.
\]
So if we take $\varepsilon >0$ such that $|(\cos 2\pi z,\sin 2\pi z) -( -1,0)| \le \delta_2$ for $|z-1/2|\le \varepsilon$ then the lemma holds, with $M= 4M_1^2| \hat{\Delta}(-1,0)|^{-1}$ (note that $|\hat{\Delta}(-1,0)|=|\Delta(-1,0)|$ is effectively computable, in fact $|\Delta(-1,0)|=\frac{\Gamma\left(\frac14\right)^4}{2^3\pi}$).
\end{proof}

To get an effective counting result for the graph of $f$ restricted to a small interval we need some way to count zeros of functions of the form $A(z,f(z))$ for non-zero polynomials $A$ in terms of the degree of $A$. For this we use a trick from \cite{BoxallJones}, which makes use of the fact that we can also ensure that the coefficients of $A$ are integers and are not too large. The trick relies on $f$ taking a suitable transcendental value at an algebraic number. We will use $f(1/2)=f_1(-1,0)$ and so we now examine this number in more detail. For brevity, we write $a=f_1(-1,0)$. By looking at the integrals defining them, we see that the values $\phi(-1)$ and $\omega_1(-1)$ are both real. So $\omega_2(-1)$ is not real, as $\omega_1(-1)$ and $\omega_2(-1)$ form a basis of a lattice in $\IC$. But the functions $f_1$ and $f_2$ are real-valued, and so by \eqref{propertyoffandg} we find that $f_2(-1,0)=0$ and so
\[
a= \frac{\phi(-1)}{\omega_1(-1)}.
\]
Note that $E_{-1}$ has $j$-invariant $1728$ and so by Lemma \ref{lem:extraauto} the point $P_{-1}$ is not torsion on $E_{-1}$. As $f_2(-1,0)=0$ is rational, $a=f_1(-1,0)$ is not. Our curve $E_{-1}$ is defined over $\IQ$ and so if $a$ were algebraic, then as a ratio of elliptic logarithms it would have to lie in the CM field of $E_{-1}$ and then it would be rational, as it is certainly real. So $a$ is transcendental. In fact it even has a good transcendence measure, as we now observe.
\begin{lemma}\label{transmeasure}
There are positive absolute constants $\mu$ and $\nu$ with the
following property. For any integers $T\ge 1$ and $N\ge 3$ and any
non-zero
polynomial $A$ of degree at most $T$ with integer coefficients of absolute value at most $N$ we have
\[
\log |A(a)| \gg -T^{\mu}(\log N)^{\nu}
\]
where the  constant  implied in Vinogradov's notation is
absolute and effective.
\end{lemma}
\begin{proof} We first find a measure for approximations of $a$ by an
algebraic number, $\alpha$ say, of degree $T$ over $\IQ$ and height
$\height{\alpha}$. By Theorem 1.5 on page 42
of \cite{DavidHirataKohno} (which is far stronger than we need) and
with
$h= max\{1,\height{\alpha}\}$  we have
\[
\log |\phi(-1) -\alpha \omega_1(-1)| \gg - T^{\mu'}h^{\nu'},
\]
for some positive absolute $\mu',\nu'$. So
\begin{equation}\label{approxmeasure}
\log |a-\alpha| \gg- T^{\mu'}h^{\nu'}.
\end{equation}

Now suppose that $A$ is a non-zero polynomial as in the statement, so of degree at most $T$, and let $\alpha$ be the root of $A$ closest to $a$. Then
\begin{equation}\label{closestroot}
|A(a)| \ge |a-\alpha|^{\deg A}.
\end{equation}
Let $B$ be the minimal polynomial of $\alpha$ over the integers. Then $B$ divides $A$ and so by Lemma 1.6.11 on page 27 of \cite{BG} we have
\begin{equation}\label{boundnormP}
|A| \ge 2^{-T} |B|,
\end{equation}
where $|A|$ and $|B|$ denote the maximum of the absolute values of the coefficients of $A$ and $B$ respectively. By Lemma 3.11 on page 80 of \cite{WaldschmidtBigBook}, $h(\alpha) \ll 1+\log |B|$ so by \eqref{boundnormP} we have
\[
h(\alpha) \ll \log N +T,
\]
as $|A|\le N$. Combining this with \eqref{closestroot} and \eqref{approxmeasure} gives
\[
\log |A(a)| \gg -T^{\mu'+1}\left(\log N+T\right)^{\nu'},
\]
which has the required form.
\end{proof}
We can now use the method from \cite{MasserZeta,BoxallJones} to prove an effective bound on the number of rational points of bounded height on the graph of a certain restriction of $f$. We use the height $H(\alpha)=\exp \height{\alpha}$.
\begin{lemma}\label{countinglemma}
There are effective absolute constants $\kappa,c>0$ such that for any integer $H\ge 3$ the number of rationals $q$ of height at most $H$ with $|q-1/2|\le \varepsilon/8$ such that $f(q)$ is also a rational of height at most $H$ is at most
$$
c(\log H)^{\kappa}.
$$
\end{lemma}
\begin{proof}
Fix an integer $H\ge 3$. In this proof only we write $f_1(z)=z+1/2$ and $f_2(z)=f(z+1/2)$. So $f_1$ is a polynomial, $f_2$ is analytic on $|z|\le \varepsilon$ and $|f_1|,|f_2| \le M$ where $M=\max\{1,M_2\}$, by Lemma \ref{continuationoff}. Let $\mathcal{Z}$ be the set of rationals $q$ with $|q|\le \varepsilon/8$ such that $f_1(q),f_2(q)$ are rationals of height at most $2H$. We show that there are at most $c(\log H)^{\kappa}$ points in $\mathcal{Z}$.  Put $Z=\varepsilon/2,d=1$ and $R=4/\varepsilon$. Then by Proposition 2 on page 2039 of \cite{MasserZeta} (with $A$ there equal to our $R$) there is a non-zero polynomial $A(X,Y)$ of degree at most $T\ll \log H$ such that $A(f_1(q),f_2(q))=0$ for all $q$ in $\mathcal{Z}$. Examining the proof of Proposition 2 (see the proof of 2.1 in \cite{BoxallJones}) we find that we can take $A$ to have integer coefficients of absolute value at most $ 2(T+1)^2(2H)^T$.

 If $A(1/2,Y)$ is the zero polynomial then we can divide $A$ by an appropriate power of $(X-1/2)$ and the resulting polynomial will still vanish at the appropriate points. So we may assume that $P(1/2,Y)$ is not the zero polynomial. By Theorem 1.1 on Page 340 of \cite{Langcomplexanalysis} (see also page 171 in \cite{Titchmarsh}) applied to the disks $|z|\le \varepsilon/8$ and $|z| \le \varepsilon/ 4$ the function $P(f_1(z),f_2(z))$ has at most
\begin{equation}\label{jensenbound}
\frac{1}{\log 2} \left(\log M_A -\log A(\frac{1}{2},a)\right)
\end{equation}
zeros in the disk $|z|\le \varepsilon/8$, where $M_A$ is a bound for the maximum of $A(f_1(z),f_2(z))$ on the disk $|z|\le \varepsilon/4$. So the cardinality of $\mathcal{Z}$ is also bounded by \eqref{jensenbound}. Using the bounds on $|f_1|$ and $|f_2|$ and on the coefficients of $P$ we find that
\[
\log M_A \ll (\log H)^2.
\]
And multiplying $A(1/2,a)$ by $2^{\deg A}$ to clear denominators we have a nonzero polynomial in $a$ with integer coefficients to which we can apply Lemma \ref{transmeasure}. We then find the cardinality of $\mathcal{Z}$ is at most
\begin{displaymath}
c(\log H)^{\kappa}
\end{displaymath}
as required.
\end{proof}
The final ingredients we need are the following bounds on the orders of torsion.
\begin{lemma}\label{galoisbounds}
Suppose that $m$ and $n$ are positive integers and that $t$ is a root of unity of order $m$ such that $P_t$ has order $n$ on $E_t$. Then
\[
m \ll [\IQ(t):\IQ]^2
\]
and
\[
n \ll [\IQ(t):\IQ]^2.
\]
\end{lemma}
\begin{proof}
The second inequality follows from David's Th\'eor\`eme 1.2(i) of \cite[Page 106]{DavidGaloisBounds} as in the proof of \cite[Lemma 5.1, Page 1685]{MZ:torsionanomalousAJ} (since $h(t)=0$) while the first inequality is classical (for example \cite[Theorem 328]{HardyWright} is much stronger than we need).
\end{proof}
We can now complete the proof of theorem \ref{thm:bc0} using the usual argument. So suppose that $t$ is a root of unity of order $m$ such that $P_t$ has order $n$ on $E_t$, for some positive integers $m$ and $n$. Let $N=\textrm{lcm} (m,n)$. Then Lemma \ref{galoisbounds} gives
\begin{equation}\label{ordervsdegree}
N \ll d^4
\end{equation}
where $d=[\IQ(t):\IQ]$. Suppose that $t^{\sigma}$ is a Galois conjugate of $t$. Then $t^{\sigma}$ is still a root of unity and $(2,.)$ is still torsion on $E_{t^{\sigma}}$. So $t^{\sigma}$ would lead to a rational point of height at most $N$ on the graph of $f$, if we had $t^{\sigma}=\exp ( 2\pi i q)$ for some rational $q$ with $|q-1/2|\le \varepsilon/8$. We might not be so lucky, but it is well known that by taking $m$ sufficiently large we can ensure that some fixed positive proportion of the conjugates of $t$ are of this form. So the number, $M$ say, of rational points of height at most $N$ on the graph of $f$ restricted to $|z-1/2| \le \varepsilon/8$ satisfies
\[
M \gg d.
\]
But by Lemma \ref{countinglemma} and \eqref{ordervsdegree} we have
\begin{eqnarray*}
M \ll (\log N)^{\kappa} 
 \ll (\log d)^{\kappa}.
\end{eqnarray*}
So $d$ is bounded above by some absolute effective constant. This completes the proof.


\bibliographystyle{amsplain}
\bibliography{CMrefs}



\end{document}